\documentclass[11pt]{amsart}      
\usepackage[a4paper,hmargin={2.5cm,2.5cm},vmargin={2.5cm,2.5cm}]{geometry}
\usepackage{geometry}               

\usepackage{graphicx}

\usepackage[dvipsnames]{xcolor}

\usepackage{amssymb,amsmath,amsthm,amsfonts}

\allowdisplaybreaks

\usepackage[english]{babel}

\usepackage[utf8]{inputenc}

\usepackage[colorlinks]{hyperref}

\hypersetup{linkcolor=blue,citecolor=blue,filecolor=black,urlcolor=blue}

\usepackage{comment}

\usepackage{mathtools}

\usepackage{bigints}

\usepackage{dsfont}

\usepackage[sc]{mathpazo}

\usepackage{tikz}
\usepackage{caption}

\newtheorem{proposition}{Proposition}[section]

\newtheorem{theorem}[proposition]{Theorem}

\newtheorem{corollary}[proposition]{Corollary}

\newtheorem{lemma}[proposition]{Lemma}

\theoremstyle{definition}

\newtheorem{definition}[proposition]{Definition}

\theoremstyle{remark}

\newtheorem{remark}[proposition]{Remark}

\numberwithin{equation}{section}

\title[Asymptotics of nonlocal nonlinear Robin energies]{Asymptotics of nonlocal nonlinear Robin energies}

\author{Serena Dipierro, Giuseppe Spadaro, and Enrico Valdinoci}
 
\date{\today. {\bf SD \& EV:} 
Department of Mathematics and Statistics,
The University of Western Australia,
35 Stirling Highway,
Perth WA 6009, Australia.
{\bf GS:} Dipartimento di Matematica e Informatica,
Università della Calabria,
87036 Arcavacata di Rende,
Italia.\\ {\tt serena.dipierro@uwa.edu.au},
{\tt giuseppe.spadaro@unical.it},
{\tt enrico.valdinoci@uwa.edu.au}.}

\begin{document}


\begin{abstract}
We study the free minimization problem for a nonlinear, nonlocal functional associated with Robin-type nonlocal exterior conditions depending on a positive parameter $\alpha$.

We prove that, for every $\alpha>0$, the minimizer exists, is unique, and satisfies suitable decay properties at infinity. We also investigate the regularity of the maps $\alpha \mapsto u_\alpha$ and $\alpha \mapsto E_\alpha$, where~$u_\alpha$ denotes the minimizer and $E_\alpha$ the corresponding energy.

Finally, we derive asymptotic expansions of the energy as $\alpha\to+\infty$ and as $\alpha\to0^+$. The latter regime requires distinguishing between two cases, depending on whether the forcing term has zero mass. In one case, the limiting energy possesses a minimizer
and $E_\alpha$ converges to its energy, but in the other case $E_\alpha$ diverges to $-\infty$.
\end{abstract}
\maketitle

\noindent


{\footnotesize 

}


{\footnotesize 
}

\section{Introduction}
\subsection{A nonlocal, nonlinear problem with Robin exterior data}
In this paper, we investigate a nonlocal Robin problem involving the fractional \(p\)-Laplacian and a nonlinear exterior condition, with particular emphasis on the asymptotic behavior of the associated energy as the Robin parameter varies.

Let $\Omega$ be a bounded open set in $\mathbb{R}^n$, with Lipschitz boundary. We focus on the study of the following nonlocal Robin problem
\begin{equation}\label{pb:RF}
    \begin{cases}
        (-\Delta)_p^s u = f & \text{in } \Omega,\\
        \alpha |u|^{q-2}u + N^p_s u =0 & \text{in } \mathbb{R}^n \setminus \overline{\Omega},
    \end{cases}
\end{equation}
where $\alpha$ is a given positive  parameter, $s\in(0,1)$,
\begin{equation*}
\begin{split}
    (-\Delta)_p^s u(x) &:= c_{n,s,p} \ PV \int_{\mathbb{R}^n} \frac{|u(x)-u(y)|^{p-2}(u(x)-u(y))}{|x-y|^{n+sp}} \ dy\\
    &=c_{n,s,p} \lim_{\varepsilon \rightarrow 0^+}\int_{\mathbb{R}^n\setminus
    B_\varepsilon(x)} \frac{|u(x)-u(y)|^{p-2}(u(x)-u(y))}{|x-y|^{n+sp}} \ dy
\end{split}
\end{equation*}
is the fractional $p$-Laplace operator and $N^p_s$ stands for the nonlocal normal $p$-derivative associated with the fractional $p$-Laplacian (see~\cite{DRV,MP}), namely
\begin{equation}\label{N_nonlocal}
    N^p_su(x): = c_{n,s,p} \int_\Omega \frac{|u(x)-u(y)|^{p-2}(u(x)-u(y))}{|x-y|^{n+sp}} \ dy, \qquad {\mbox{for all }} x \in \mathbb{R}^n \setminus \overline{\Omega}.
\end{equation}
The parameters $p > 1$ and $q > 1$ determine the nonlinear nature of the problem, both inside the domain $\Omega$ and in the exterior region $\mathbb{R}^n \setminus \overline{\Omega}$.

Given a measurable function $u : \mathbb{R}^n \rightarrow \mathbb{R}$, we set
\begin{equation}\label{NORMA}\begin{split}
[u]_{W^{s,p}(\Omega)}&:=\left(
 \int_{\mathbb{R}^{2n}\setminus (\mathbb{R}^n\setminus \Omega)^2} \frac{|u(x)-u(y)|^p}{|x-y|^{n+sp}} \ dxdy\right)^{\frac1p}\\
 {\mbox{and }}\quad
    \|u\|_{W^{s,p}(\Omega)} &:= \left(\|u\|_{L^p(\Omega)}^p  +
    [u]_{W^{s,p}(\Omega)}^p\right)^{\frac1p}\end{split}
\end{equation}
and we define
\begin{equation}\label{space}
    W^{s,p}(\Omega) := \{u:\mathbb{R}^n \rightarrow \mathbb{R} \ \text{ measurable}: \|u\|_{W^{s,p}(\Omega)} < +\infty\}.
\end{equation}
We further consider the following functional space, which is natural in the context of our framework,
\begin{equation*}
    \mathcal{W}^{s,p} := W^{s,p}(\Omega) \cap L^q(\mathbb{R}^n \setminus \Omega),
\end{equation*}
endowed with the norm
\begin{equation*}
    \|u\|_{\mathcal{W}^{s,p}} := \|u\|_{W^{s,p}(\Omega)} + \|u\|_{L^q(\mathbb{R}^n \setminus \Omega)}.
\end{equation*}

Furthermore, we will adopt the following definition of weak solution in this setting.
\begin{definition}\label{DEBOLE}
    Let $f\in (\mathcal{W}^{s,p})^*$. A function $u \in \mathcal{W}^{s,p}$ is a weak solution of \eqref{pb:RF} if it solves
    \begin{equation*}
    \begin{split}
       &\frac{c_{n,s,p}}{2} \int_{\mathbb{R}^{2n}\setminus (\mathbb{R}^n\setminus \Omega)^2} \frac{|u(x)-u(y)|^{p-2}(u(x)-u(y))(\varphi(x)-\varphi(y))}{|x-y|^{n+sp}} \ dxdy + \alpha \int_{\mathbb{R}^n \setminus \Omega} |u|^{q-2}u\varphi \ dx\\
       &\qquad\qquad= \int_\Omega f \varphi \ dx,
    \end{split}
    \end{equation*}
    for all $\varphi \in \mathcal{W}^{s,p}$.
\end{definition}

We recall that the notation~$(\mathcal{W}^{s,p})^*$ stands for the
dual space of~$\mathcal{W}^{s,p}$.
Also, for clarity, we note that, since the source term $f$ belongs to~$(\mathcal{W}^{s,p})^*$, the integral $\int_\Omega f \varphi \, dx$ should be understood as the duality pairing between $f$ and $\varphi$ in $\mathcal{W}^{s,p}$.

In this work, we study the problem from a variational perspective by introducing a suitable energy functional, whose critical points correspond to weak solutions of \eqref{pb:RF}, namely
\begin{equation}\label{FUNCTIONAL}
    J_\alpha(u) := \frac{c_{n,s,p}}{2p} \int_{\mathbb{R}^{2n}\setminus (\mathbb{R}^n\setminus \Omega)^2} \frac{|u(x)-u(y)|^p}{|x-y|^{n+sp}} \ dxdy + \frac{\alpha}{q} \int_{\mathbb{R}^n \setminus \Omega} |u|^q \ dx - \int_\Omega f u \ dx,
\end{equation}
for every $u \in \mathcal{W}^{s,p}$.

The main goal of this paper is to analyze the asymptotic behavior of the minimal energy level
\begin{equation}\label{EALFADE}
E_\alpha := \inf_{u \in \mathcal{W}^{s,p}} \left\{J_\alpha (u)\right\},
\end{equation}
as the parameter \(\alpha\) varies. 
One could expect that, if \(\alpha \to +\infty\), the term~$\alpha |u|^{q-2}u$
in the exterior condition of~\eqref{pb:RF}
dominates and the condition effectively imposes a Dirichlet-type constraint on the exterior domain, whereas if $\alpha \to 0^+$, the influence of this  term becomes negligible and, under the compatibility assumption $\int_\Omega f dx =0$, the Robin condition approaches a Neumann-type behavior.

Our goal is to describe rigorously how the relative strength of the two terms in the Robin condition determines the limiting regimes and how the nonlocal operator $N^p_s u$ affects the solution in the exterior domain. We stress that our results are also new in the linear case~$p=q=2$. Indeed, to the best of our knowledge, no analogous results are available in the nonlocal framework, even for quadratic energies.

In fact, as far as we know, problem~\eqref{pb:RF} was not considered before, and
the asymptotic behavior of Robin-type energies in the limits $\alpha \to 0^+$ and $\alpha \to +\infty$ has so far been investigated only in the local setting, see~\cite{BBBT, BD, BO}. In particular, in~\cite{BO}, the authors consider nonlinear Robin energies of the form
\begin{equation*}
    \mathcal{E}_\alpha(u) = \frac{1}{p}\int_\Omega |\nabla u|^p dx + \frac{\alpha}{q}\int_{\partial \Omega} |u|^q \ d \mathcal{H}^{n-1} - \int_\Omega fu \ dx,
\end{equation*}
and study their asymptotic behavior as the parameter $\alpha$ varies. They show that, in the limits $\alpha \to 0^+$ and $\alpha \to +\infty$, the corresponding energies converge to the Neumann and Dirichlet energies respectively, and they further derive a first-order asymptotic expansion of the minimum energy in both regimes.

In this spirit, our paper can be seen as the nonlocal counterpart of the analysis in~\cite{BO}, though the presence of the nonlocal operator here results in an analysis that deviates from the classical case.
Below, we summarize the main results of this paper.

\subsection{The unique solution for~$\alpha>0$}
To start with, we point out that problem~\eqref{pb:RF} is mathematically coherent:

\begin{theorem}\label{PF:MA1}
For any $\alpha >0$, there exists a unique weak solution~$u_\alpha$ of \eqref{pb:RF}.

Moreover:
\begin{enumerate}
\item The function~$u_\alpha$ realizes the infimum in~\eqref{EALFADE}.
\item If~$q\geq p$, then
\begin{equation}\label{dePSK:DE}
    \limsup_{|x|->+\infty} |x|^\frac{n+sp}{q-1}|u_\alpha(x)| < +\infty
\end{equation}
and
  \begin{equation}\label{LINTE}
        \int_{\mathbb{R}^n \setminus \Omega} |u_\alpha|^{q-2}u_\alpha \ dx= \frac{1}{\alpha} \int_\Omega f \ dx.
    \end{equation}
    \item The map $\alpha \mapsto u_\alpha$ is continuous from $(0,+\infty)$ to $\mathcal{W}^{s,p}$.
    \item The map $\alpha \mapsto E_\alpha$ is of class $C^1$ and
    \begin{equation}\label{2PSK:-1eidk}
        \frac{d E_\alpha}{d\alpha} = \frac{1}{q} \int_{\mathbb{R}^n\setminus \Omega} |u_\alpha|^q \ dx.
    \end{equation}
\end{enumerate}
\end{theorem}

We will devote the forthcoming Section~\ref{P}
to the proof of Theorem~\ref{PF:MA1}
(the simple, but technical, proof that any minimizer $u_\alpha$ of the functional $J_\alpha$ satisfies \eqref{pb:RF} in the sense of Definition~\ref{DEBOLE} being deferred to Appendix~\ref{A}
for the convenience of the reader).

We also investigate the asymptotic behavior of $E_\alpha$ as $\alpha \to +\infty$. In particular, we show that the functional $E_\alpha$ converges to the limiting energy
\begin{equation*}
    E_\infty : = \inf_{u\in W^{s,p}_0(\Omega)} \left\{ \frac{c_{n,s,p}}{2p} \int_{\mathbb{R}^{2n}\setminus (\mathbb{R}^n\setminus \Omega)^2} \frac{|u(x)-u(y)|^p}{|x-y|^{n+sp}} \ dxdy - \int_\Omega f u \ dx\right\},
\end{equation*}
whose unique minimizer $u_\infty \in W^{s,p}_0(\Omega)$ weakly solves the nonlocal Dirichlet problem
    \begin{equation*}
    \begin{cases}
        (-\Delta)_p^s u_\infty = f & \text{in } \Omega,\\
        u_\infty =0 & \text{in } \mathbb{R}^n \setminus \overline \Omega.
    \end{cases}
\end{equation*}
In fact, we prove a sharper and more quantitative result, by establishing the following asymptotic expansion of the energy $E_\alpha$ as $\alpha \rightarrow +\infty$.

\begin{theorem}\label{D_exp}
    Assume that $N^p_s u_\infty \in L^\frac{q}{q-1}( \mathbb{R}^n \setminus \Omega)$. Then,
    \begin{equation*}
        E_\alpha=E_\infty + \alpha^{-\frac{1}{q-1}}\frac{1-q}{q} \int_{\mathbb{R}^n \setminus \Omega}|N^p_su_\infty|^\frac{q}{q-1} \ dx + o(\alpha^{-\frac{1}{q-1}}), \qquad \text{as } \alpha \rightarrow + \infty.
    \end{equation*}
\end{theorem}

The proof of Theorem~\ref{D_exp} is contained in Section~\ref{D}.

We also study the behavior of the functional $E_\alpha$ as $\alpha \to 0^+$. 
To this end, we need to distinguish two cases, depending 
on whether or not the mass of the forcing term is zero.

More specifically, under the compatibility condition
\begin{equation}\label{COMPACO}
    \int_\Omega f  \ dx =0,
\end{equation}
the minimization problem
\begin{equation*}
    E_0 :=\inf_{u \in W^{s,p}(\Omega)} \left\{ \frac{c_{n,s,p}}{2p} \int_{\mathbb{R}^{2n}\setminus (\mathbb{R}^n\setminus \Omega)^2} \frac{|u(x)-u(y)|^p}{|x-y|^{n+sp}} \ dxdy - \int_\Omega f u \ dx\right\}
\end{equation*}
is attained by a function $u_0 \in W^{s,p}(\Omega)$, which is unique up to additive constants
and weakly solves the nonlocal Neumann problem
\begin{equation*}
    \begin{cases}
        (-\Delta)_p^s u_0 = f & \text{in } \Omega,\\
        N^p_s u_0 =0 & \text{in } \mathbb{R}^n \setminus \overline \Omega.
    \end{cases}
\end{equation*}
In this situation,
we show that $E_\alpha$ converges to~$E_0$, according to the following result:

\begin{theorem}\label{N_main}
Assume~\eqref{COMPACO}.
  Then,
    \begin{equation*}
        \lim_{\alpha \rightarrow 0^+} E_\alpha = E_0.
    \end{equation*}
\end{theorem}

If, instead, the compatibility condition in~\eqref{COMPACO}
is not satisfied, we prove that the energy $E_\alpha$ becomes unbounded from below as $\alpha \to 0^+$, namely:

\begin{theorem}\label{int_f_div}
Assume that
$$    \int_\Omega f  \ dx \ne0.$$
Then,
    \begin{equation*}
        \lim_{\alpha \rightarrow 0^+} E_\alpha = - \infty.
    \end{equation*}
\end{theorem}

We will present in Section~\ref{N} the proofs of Theorems~\ref{N_main} and~\ref{int_f_div}.

\section{Preliminary results and proof of Theorem~\ref{PF:MA1}}\label{P}
In this section, we establish several properties of weak solutions
to~\eqref{pb:RF}, along with properties of the map \(\alpha \mapsto E_\alpha\), thus establishing
Theorem~\ref{PF:MA1}.

The following results establish a weak maximum principle for weak supersolutions of \eqref{pb:RF} and a weak comparison principle for weak subsolutions and supersolutions of \eqref{pb:RF}.

\begin{theorem}\label{TPM}
    Let $u \in \mathcal{W}^{s,p}$ weakly satisfy
    \begin{equation}\label{MAX}
        \begin{cases}
            (-\Delta)_p^su \geq 0 & \text{ in } \Omega,\\
            \alpha |u|^{q-2}u + N_s^pu \geq 0 & \text{ in } \mathbb{R}^n \setminus\overline{\Omega}.
        \end{cases}
    \end{equation}
    Then, $u \geq 0$ almost everywhere in $\mathbb{R}^n$.
\end{theorem}
\begin{proof}
Using $0 \leq u^-(x):= \max\{-u(x),0\}\in \mathcal{W}^{s,p}$ as a test function in the weak formulation of~\eqref{MAX}, we get
\begin{equation}\label{PM}
\begin{split}
    \frac{c_{n,s,p}}{2} &\int_{\mathbb{R}^{2n}\setminus (\mathbb{R}^n\setminus \Omega)^2} \frac{|u(x) - u(y)|^{p-2}(u(x) - u(y))(u^-(x)-u^-(y))}{|x-y|^{n+sp}} \ dxdy\\
    &\qquad\qquad\qquad- \int_{\mathbb{R}^n \setminus \Omega} N_s^pu(x)u^-(x)\ dx\geq0.
\end{split}
\end{equation}
Moreover, since $-N_s^pu(x)\leq \alpha |u(x)|^{q-2}u(x)$ in
$\mathbb{R}^n\setminus\Omega$ and $u^-(x)\geq0$, we have
\begin{equation*}
-\int_{\mathbb{R}^n\setminus\Omega}
N_s^pu(x)u^-(x)\,dx
\leq
\alpha
\int_{\mathbb{R}^n\setminus\Omega}
|u(x)|^{q-2}u(x)u^-(x)\,dx.
\end{equation*}
Therefore, from \eqref{PM}, it follows that
\begin{equation}\label{2.5BIS}
\begin{split}
  &  \frac{c_{n,s,p}}{2}
    \int_{\mathbb{R}^{2n}\setminus (\mathcal{C}\Omega)^2}
    \frac{|u(x) - u(y)|^{p-2}(u(x)-u(y))(u^-(x)-u^-(y))}
    {|x-y|^{n+sp}}
    \,dxdy\\
    &\qquad\qquad\qquad+
    \alpha
    \int_{\mathbb{R}^n\setminus\Omega}
    |u(x)|^{q-2}u(x)u^-(x)\,dx
    \geq0.
\end{split}
\end{equation}

Now, by inspection, one sees that,
for a.e. $(x,y) \in \mathbb{R}^{2n}\setminus (\mathbb{R}^n\setminus \Omega)^2$,
\begin{equation*}
    |u(x) - u(y)|^{p-2}(u(x)-u(y))(u^-(x) -u^-(y)) \leq -|u^-(x) -u^-(y)|^p.
\end{equation*}
Hence,
\begin{equation}\label{2.5TER}
\begin{split}
    &\frac{c_{n,s,p}}{2} \int_{\mathbb{R}^{2n}\setminus (\mathbb{R}^n\setminus \Omega)^2} \frac{|u(x) - u(y)|^{p-2}(u(x) - u(y))(u^-(x)-u^-(y))}{|x-y|^{n+sp}} \ dxdy\\
    &\qquad\qquad\leq -\frac{c_{n,s,p}}{2} \int_{\mathbb{R}^{2n}\setminus (\mathbb{R}^n\setminus \Omega)^2} \frac{|u^-(x) - u^-(y)|^{p}}{|x-y|^{n+sp}} \ dxdy.
\end{split}
\end{equation}

Moreover, it is easily verified that
\begin{equation*}
    |u(x)|^{q-2}u(x)u^-(x) = -|u^-(x)|^q, \quad \text{for a.e. } x \in \mathbb{R}^n \setminus \Omega.
\end{equation*}
Consequently,
\begin{equation}\label{2.5QUA}
    \alpha \int_{\mathbb{R}^n \setminus \Omega} |u(x)|^{q-2}u(x)u^-(x)\ dx = -\alpha \int_{\mathbb{R}^n \setminus \Omega} |u^-(x)|^q\ dx.
\end{equation}

Combining~\eqref{2.5BIS}, \eqref{2.5TER} and~\eqref{2.5QUA}, we obtain
\begin{equation*}
   0 \leq -\frac{c_{n,s,p}}{2} \int_{\mathbb{R}^{2n}\setminus (\mathbb{R}^n\setminus \Omega)^2} \frac{|u^-(x)-u^-(y)|^p}{|x-y|^{n+sp}} \ dxdy -\alpha \int_{\mathbb{R}^n \setminus \Omega} |u^-(x)|^q\ dx \leq 0.
\end{equation*}
This entails that
\begin{equation*}
    \int_{\mathbb{R}^{2n}\setminus (\mathbb{R}^n\setminus \Omega)^2} \frac{|u^-(x)-u^-(y)|^p}{|x-y|^{n+sp}} \ dxdy =0 \quad \text{and} \quad \int_{\mathbb{R}^n \setminus \Omega} |u^-(x)|^q\ dx=0.
\end{equation*}
The first identity implies that $u^-$ is almost everywhere equal to a
constant in $\mathbb{R}^n$. The second one yields
$u^-=0$ almost everywhere in $\mathbb{R}^n\setminus\Omega$; hence such a
constant must be zero. Consequently, $u^- =0$ a.e. in $\mathbb{R}^n$, that is $u \geq 0$ a.e. in $\mathbb{R}^n$.
\end{proof}
\begin{theorem}\label{THM:UNI}
    Let $u_1$, $u_2 \in \mathcal{W}^{s,p}$ weakly satisfy
    \begin{equation}\label{PC}
        \begin{cases}
            (-\Delta)_p^su_1 \leq (-\Delta)_p^su_2 & \text{ in } \Omega,\\
            \alpha |u_1|^{q-2}u_1 + N_s^p u_1 \leq \alpha |u_2|^{q-2}u_2 + N_s^p u_2 & \text{ in } \mathbb{R}^n \setminus\overline{\Omega}.
        \end{cases}
    \end{equation}
    Then, $u_1 \leq u_2$ almost everywhere in $\mathbb{R}^n$.
\end{theorem}

\begin{proof} Let~$w:=u_2-u_1$.
Using $0\leq w^-= (u_2-u_1)^-$ as a test function in the weak formulation of~\eqref{PC}, we obtain
\begin{equation*}
    \begin{split}
    \frac{c_{n,s,p}}{2} &\int_{\mathbb{R}^{2n}\setminus (\mathbb{R}^n\setminus \Omega)^2} \frac{[A(u_2(x),u_2(y))-A(u_1(x),u_1(y)))](w^-(x)-w^-(y))}{|x-y|^{n+sp}} \ dxdy\\
    &\qquad\qquad+\alpha \int_{\mathbb{R}^n \setminus \Omega}(|u_2|^{q-2}u_2-|u_1|^{q-2}u_1) w^-\ dx\geq0,
\end{split}
\end{equation*}
where
\begin{equation*}
    A(\lambda,\tau):= |\lambda-\tau|^{p-2}(\lambda-\tau), \quad {\mbox{ for all }}\lambda,\tau \in \mathbb{R}.
\end{equation*}

Consequently,
\begin{equation*}
    \begin{split}
    \frac{c_{n,s,p}}{2} &
    \int_{ {\mathbb{R}^{2n}\setminus (\mathbb{R}^n\setminus \Omega)^2} \atop{\{u_2 < u_1\}\times \{u_2 < u_1\}}}
\frac{[A(u_2(x),u_2(y))-A(u_1(x),u_1(y))](u_1(x)-u_1(y)-u_2(x)+u_2(y))}{|x-y|^{n+sp}} \ dxdy\\
    &-\frac{c_{n,s,p}}{2}
    \int_{ {\mathbb{R}^{2n}\setminus (\mathbb{R}^n\setminus \Omega)^2} \atop{\{u_2 \geq u_1\}\times \{u_2 < u_1\}}}
    \frac{[A(u_2(x),u_2(y))-A(u_1(x),u_1(y))]w^-(y)}{|x-y|^{n+sp}} \ dxdy\\
    &+\frac{c_{n,s,p}}{2}
    \int_{ {\mathbb{R}^{2n}\setminus (\mathbb{R}^n\setminus \Omega)^2} \atop{\{u_2 < u_1\}\times \{u_2 \geq u_1\}}}
    \frac{[A(u_2(x),u_2(y))-A(u_1(x),u_1(y))]w^-(x)}{|x-y|^{n+sp}} \ dxdy\\
    &\qquad\qquad+\alpha \int_{\mathbb{R}^n \setminus \Omega}(|u_2|^{q-2}u_2-|u_1|^{q-2}u_1) w^-\ dx\geq0.
\end{split}
\end{equation*}
By the properties of the function $w^-$ and by the monotonicity of the map $t \mapsto |t|^{k-2}t$, with~$k>1$, we deduce that all the terms in the previous inequality have to be equal to $0$. This implies that $w^- = 0$ a.e. in $\{u_2 < u_1\}$.
\end{proof}

We then obtain a uniqueness result:

\begin{lemma}\label{LEM:UNI}
Problem~\eqref{pb:RF} admits at most one weak solution.
\end{lemma}

\begin{proof}
The uniqueness statement follows from Theorem~\ref{THM:UNI}.
\end{proof}

We now turn our attention to the functional $E_\alpha$
introduced in~\eqref{EALFADE}. In order to prove the existence of minimizers of the functional $J_\alpha$, the following result is needed.
\begin{lemma}\label{coercivity}
    For all $u \in \mathcal{W}^{s,p}$, there exists a positive constant $C(n,\Omega,s,p,q)$ such that
    \begin{equation}\label{INEQ}
        \|u\|_{L^p(\Omega)} \leq C \left[  \left(\int_{\mathbb{R}^{2n}\setminus (\mathbb{R}^n\setminus \Omega)^2} \frac{|u(x)-u(y)|^p}{|x-y|^{n+sp}} \ dxdy\right)^{\frac{1}{p}} + \left(\int_{\mathbb{R}^n \setminus \Omega} |u|^q \ dx \right)^{\frac{1}{q}} \right].
    \end{equation}
    
    Moreover, for any $\alpha>0$, the functional $J_\alpha$ is coercive on $\mathcal{W}^{s,p}$.
\end{lemma}

\begin{proof} To prove~\eqref{INEQ},
we argue by contradiction, supposing that \eqref{INEQ} does not hold true. Therefore, without loss of generality, we can suppose that there exists a sequence $\{u_k\}_k \subset \mathcal{W}^{s,p}$ such that 
\begin{equation}\label{cnsar3t7rfguawegfuiSDFGdiweugfuke76}
\|u_k\|_{L^p(\Omega)}=1\end{equation} and
    \begin{equation}\label{conv_ass}
        \left(\int_{\mathbb{R}^{2n}\setminus (\mathbb{R}^n\setminus \Omega)^2} \frac{|u_k(x)-u_k(y)|^p}{|x-y|^{n+sp}} \ dxdy\right)^{\frac{1}{p}} + \left(\int_{\mathbb{R}^n \setminus \Omega} |u_k|^q \ dx \right)^{\frac{1}{q}} \rightarrow 0, \ \text{ as } k \rightarrow\infty.
    \end{equation}
    
    We notice that the sequence $\{u_k\}$ is uniformly bounded in $\mathcal{W}^{s,p}$, thus there exists a function~$v \in \mathcal{W}^{s,p}$ such that, up to a subsequence, as~$k\to+\infty$,
\begin{equation*}
  {\mbox{$u_k \rightharpoonup v$ in~$\mathcal{W}^{s,p}$ and~$u_k \rightarrow v$ in~$L^p(\Omega)$ and almost everywhere in~$\Omega$,}}
\end{equation*}
see e.g.~\cite[Theorem~7.2]{GUIDA} for the compact embedding.

From~\eqref{conv_ass}, we also infer that
\begin{equation}\label{vcmsery43t673rqwvdfhcwkPIUY}{\mbox{$u_k\rightarrow0$ almost everywhere in~$\mathbb{R}^n\setminus\Omega$, as~$k\to+\infty$.}}
\end{equation}

As a consequence, by \eqref{conv_ass} and Fatou's lemma,
\begin{equation*}
    \int_{\mathbb{R}^{2n}\setminus (\mathbb{R}^n\setminus \Omega)^2} \frac{|v(x)-v(y)|^p}{|x-y|^{n+sp}} \ dxdy = 0,
\end{equation*}
and therefore $v$ is constant a.e. in $\mathbb{R}^n$. 
This and~\eqref{vcmsery43t673rqwvdfhcwkPIUY} entail that~$v=0$
a.e. in $\mathbb{R}^n$. 

However, we know that $u_k \rightarrow v$ in $L^p(\Omega)$ and thus we get a contradiction with~\eqref{cnsar3t7rfguawegfuiSDFGdiweugfuke76}. The inequality in~\eqref{INEQ} is thereby established.

Now, to prove the coercivity of $J_\alpha$, assume by contradiction that there exists a sequence~$\{u_k\}_k \subset \mathcal{W}^{s,p}$ such that
\begin{equation*}
    \|u_k\|_{\mathcal{W}^{s,p}} \rightarrow + \infty \quad \text{ and } \quad J_\alpha(u_k) \leq C.
\end{equation*}
By the definition of $J_\alpha$, see \eqref{FUNCTIONAL}, it follows that
\begin{equation}\label{stima_prima}
\begin{split}
    \frac{c_{n,s,p}}{2p} \int_{\mathbb{R}^{2n}\setminus (\mathbb{R}^n\setminus \Omega)^2} \frac{|u_k(x)-u_k(y)|^p}{|x-y|^{n+sp}} \ dxdy + \frac{\alpha}{q} \int_{\mathbb{R}^n \setminus \Omega} |u_k|^q \ dx &\leq C + \int_\Omega f u_k \ dx\\
    &\leq C + \|f\|_{(\mathcal{W}^{s,p})^*}\|u_k\|_{\mathcal{W}^{s,p}},
\end{split}
\end{equation} up to renaming~$C>0$.

Moreover, see \eqref{NORMA}, we deduce from \eqref{INEQ} that 
\begin{equation*}
    \|u_k\|_{\mathcal{W}^{s,p}} \leq C \left[  \left(\int_{\mathbb{R}^{2n}\setminus (\mathbb{R}^n\setminus \Omega)^2} \frac{|u_k(x)-u_k(y)|^p}{|x-y|^{n+sp}} \ dxdy\right)^{\frac{1}{p}} + \left(\int_{\mathbb{R}^n \setminus \Omega} |u_k|^q \ dx \right)^{\frac{1}{q}} \right].
\end{equation*}
Therefore, using this information into~\eqref{stima_prima},
we find that
\begin{equation*}
\begin{split}
    \frac{c_{n,s,p}}{2p} &\int_{\mathbb{R}^{2n}\setminus (\mathbb{R}^n\setminus \Omega)^2} \frac{|u_k(x)-u_k(y)|^p}{|x-y|^{n+sp}} \ dxdy + \frac{\alpha}{q} \int_{\mathbb{R}^n \setminus \Omega} |u_k|^q \ dx\\
    &\leq C + C\|f\|_{(\mathcal{W}^{s,p})^*} \left[  \left(\int_{\mathbb{R}^{2n}\setminus (\mathbb{R}^n\setminus \Omega)^2} \frac{|u_k(x)-u_k(y)|^p}{|x-y|^{n+sp}} \ dxdy\right)^{\frac{1}{p}} + \left(\int_{\mathbb{R}^n \setminus \Omega} |u_k|^q \ dx \right)^{\frac{1}{q}} \right].
\end{split}
\end{equation*}
Hence, since~$p$, $q>1$,
we conclude that $\{u_k\}$ is uniformly bounded in $\mathcal{W}^{s,p}$, in contradiction with the assumption $\|u_k\|_{\mathcal{W}^{s,p}} \rightarrow + \infty$.
\end{proof}

The following result provides the existence and uniqueness of a minimizer $u_\alpha \in \mathcal{W}^{s,p}$ of the functional $J_\alpha$, for any given $\alpha > 0$.

\begin{lemma}\label{esun}
    For any $\alpha > 0$, there exists a unique function $u_\alpha \in \mathcal{W}^{s,p}$ achieving $E_\alpha$. 
    
    Moreover, $u_\alpha$ is a weak solution of \eqref{pb:RF}, in the sense of Definition~\ref{DEBOLE}.
\end{lemma}

\begin{proof}
We know that, by Lemma~\ref{coercivity}, the functional $J_\alpha$ is coercive. Moreover, we claim that~$J_\alpha$ is weakly lower semicontinuous. Indeed, since the Gagliardo seminorm defines a convex and continuous functional on ${W}^{s,p}(\Omega)$, it follows that it is weakly lower semicontinuous with respect to the weak topology of ${W}^{s,p}(\Omega)$.
In addition, the term involving the $L^q$-norm on~$\mathbb{R}^n \setminus \Omega$ is also convex and continuous, hence it is weakly lower semicontinuous in $L^q(\mathbb{R}^n \setminus \Omega)$. In addition, the term containing the source function $f$ is linear and continuous, therefore it is weakly continuous.
Combining these observations, we conclude that $J_\alpha$ is weakly lower semicontinuous on $\mathcal{W}^{s,p}$.

Let $\{u_k\}_k\subset \mathcal{W}^{s,p}$ be a minimizing sequence for $J_\alpha$. By coercivity, $\{u_k\}_k$ is bounded in $\mathcal{W}^{s,p}$. 
Hence, there exists a function $u_\alpha \in \mathcal{W}^{s,p}$ such that, up to a subsequence, $u_k \rightharpoonup u_\alpha$ weakly in~$\mathcal{W}^{s,p}$. By weak lower semicontinuity, it follows that $u_\alpha$ is a minimizer of $J_\alpha$. Furthermore, the functional $J_\alpha$ is strictly convex, as the sum of two strictly convex terms and a linear one, hence the minimizer $u_\alpha$ is unique.

In addition, by Proposition~\ref{minimo_weak}, it follows that the unique minimizer $u_\alpha \in \mathcal{W}^{s,p}$ of the functional $J_\alpha$ is a weak solution of \eqref{pb:RF}.
\end{proof}

The following result provides, for $q \geq p$, a decay estimate for the unique weak solution~$u_\alpha \in \mathcal{W}^{s,p}$ of problem \eqref{pb:RF} as $|x| \to +\infty$.
\begin{lemma}\label{DECAY:TH}
For any $\alpha >0$, let $u_\alpha \in \mathcal{W}^{s,p}$ be the unique weak solution of \eqref{pb:RF}, with $q\geq p$. Then,
\begin{equation*}
    \limsup_{|x|->+\infty} |x|^\frac{n+sp}{q-1}|u_\alpha(x)| < +\infty.
\end{equation*}
\end{lemma}

\begin{proof} Since $\Omega$ is bounded, there exists $R>0$ such that $\Omega \subset B_{R}$. Without loss of generality, assume $|x| > 2R$ and $y \in \Omega$. Then, it follows that
    \begin{equation*}
        |x-y| \geq |x|-|y| \geq |x|-R \geq \frac{|x|}{2}.
    \end{equation*}
    Therefore,
    \begin{equation*}
       \frac1{ |x-y|^{n+sp}} \leq \frac{2^{n+sp}}{|x|^{n+sp}}.
    \end{equation*}
    By \eqref{N_nonlocal}, for any $x \in \mathbb{R}^n \setminus \overline{\Omega}$ such that $|x|>2R$, we obtain
    \begin{equation*}
    \begin{split}
        |N^p_s u_\alpha(x)| &= c_{n,s,p} \left |\int_\Omega \frac{|u_\alpha(x)-u_\alpha(y)|^{p-2}(u_\alpha(x)-u_\alpha(y))}{|x-y|^{n+sp}} \ dy\right|\\
        &\leq \frac{c_{n,s,p}2^{n+sp}\max\{1,2^{p-2}\}}{|x|^{n+sp}}\left( \int_\Omega |u_\alpha(x)|^{p-1}\ dy + \int_\Omega |u_\alpha(y)|^{p-1} \ dy\right)\\
        &\leq \frac{c_{n,s,p}2^{n+sp}\max\{1,2^{p-2}\}}{|x|^{n+sp}}\left(|\Omega||u_\alpha(x)|^{p-1} +\|
        u_\alpha\|^{p-1}_{L^{p-1}(\Omega)} \right).
    \end{split}
    \end{equation*}
    
    Since $u_\alpha$ is a weak solution of \eqref{pb:RF}, we recall that $\alpha |u_\alpha(x)|^{q-2}u_\alpha(x) + N^p_s u_\alpha(x) =0$ for any~$x \in \mathbb{R}^n \setminus \overline{\Omega}$. Therefore,
    \begin{equation*}
        \alpha |u_\alpha(x)|^{q-1} = |N^p_s u_\alpha(x)| \leq \frac{c_{n,s,p}2^{n+sp}\max\{1,2^{p-2}\}}{|x|^{n+sp}}\left(|\Omega||u_\alpha(x)|^{p-1} +
        \|u_\alpha\|^{p-1}_{L^{p-1}(\Omega)} \right).
    \end{equation*}
    In particular,
    \begin{equation*}
    |u_\alpha(x)|^{q-1} - \frac{c_1}{\alpha|x|^{n+sp}}|u_\alpha(x)|^{p-1} \leq \frac{c_2}{\alpha|x|^{n+sp}},
    \end{equation*}
    for some~$c_1$, $c_2>0$, depending on~$n$, $s$, $p$, $\Omega$,
    and~$\|u_\alpha\|_{L^{p-1}(\Omega)}$.
    
    It is straightforward to verify that, for $q \geq p$,
    \begin{equation*}
        |u_\alpha(x)|^{p-1} \leq 1+ |u_\alpha(x)|^{q-1}.
    \end{equation*}
    Therefore, we deduce that
    \begin{equation*}
        \left(1 - \frac{c_1}{\alpha|x|^{n+sp}}\right)|u_\alpha(x)|^{q-1} \leq \frac{c_2}{\alpha|x|^{n+sp}},
    \end{equation*} up to renaming~$c_2$.
    
    Hence, for $|x|$ large enough, it follows that
    \begin{equation*}
        |u_\alpha(x)|\leq \frac{C}{|x|^{\frac{n+sp}{q-1}}},
    \end{equation*}
    where $C=C(n,s,p,q,|\Omega|,\alpha,\|u_\alpha\|_{L^{p-1}(\Omega)})$ is a positive constant.
\end{proof}

\begin{remark}\label{integr}
    We notice that, for any $\alpha >0$, the unique weak solution $u_\alpha$ of \eqref{pb:RF}, with $q \geq p$, is actually in $L^{q-1}(\mathbb{R}^n \setminus \Omega)$. Indeed, this follows from the fact that $u_\alpha \in \mathcal{W}^{s,p}$ and by the decay at infinity proved in Lemma~\ref{DECAY:TH}.
\end{remark}
\begin{lemma}\label{LINTE:L}
    For any $\alpha > 0$, let $u_\alpha \in \mathcal{W}^{s,p}$ be the unique weak solution of \eqref{pb:RF} with $q\geq p$. Then,
    \begin{equation*}
        \int_{\mathbb{R}^n \setminus \Omega} |u_\alpha|^{q-2}u_\alpha \ dx= \frac{1}{\alpha} \int_\Omega f \ dx.
    \end{equation*}
\end{lemma}
\begin{proof}
Since $\Omega$ is bounded, it follows that there exists $R_0>0$ such that $\Omega \subset B_{R_0}$.

    Let~$R>2R_0$ and let~$\phi_R \in C_c^\infty(\mathbb{R}^n)$ be a standard cutoff function such that $0\leq \phi_R \leq 1$ and
    \begin{equation*}
        \begin{cases}
            \phi_R=1 & \text{in } B_{R},\\
            \phi_R=0 & \text{in } B_{2R}^c.
        \end{cases}
    \end{equation*}
Using $\phi_R$ as a test function in the weak formulation of~\eqref{pb:RF}, we get
    \begin{equation}\label{weak_pass}
\begin{split}
        &\frac{c_{n,s,p}}{2} \int_{\mathbb{R}^{2n}\setminus (\mathbb{R}^n\setminus \Omega)^2} \frac{|u_\alpha(x) - u_\alpha(y)|^{p-2}(u_\alpha(x) - u_\alpha(y))(\phi_R(x)-\phi_R(y))}{|x-y|^{n+sp}} \ dxdy\\
& \qquad\qquad\qquad+\alpha \int_{\mathbb{R}^n \setminus \Omega} |u_\alpha|^{q-2}u_\alpha \phi_R\ dx = \int_\Omega f \phi_R dx.
\end{split}
    \end{equation}
    
    Since $\phi_R(x)=\phi_R(y)=1$ for any $x,y \in \Omega$, we obtain
    \begin{equation*}
        \int_{\Omega \times \Omega} \frac{|u_\alpha(x) - u_\alpha(y)|^{p-2}(u_\alpha(x) - u_\alpha(y))(\phi_R(x)-\phi_R(y))}{|x-y|^{n+sp}} \ dxdy=0
    \end{equation*}
    and
    \begin{equation*}
        \int_\Omega f\phi_Rdx= \int_\Omega fdx.
    \end{equation*}
    
    We now estimate the mixed term in~\eqref{weak_pass}. By H\"older's inequality, it follows that
    \begin{equation*}
    \begin{split}
        &\left|\int_{\Omega\times (\mathbb{R}^n \setminus \Omega)}\frac{|u_\alpha(x) - u_\alpha(y)|^{p-2}(u_\alpha(x) - u_\alpha(y))(\phi_R(x)-\phi_R(y))}{|x-y|^{n+sp}} \ dxdy \right|\\
        & \qquad\qquad\leq [u_\alpha]_{W^{s,p}}^{p-1}\left(
        \int_{\Omega\times (\mathbb{R}^n \setminus \Omega)}\frac{|\phi_R(x)-\phi_R(y)|^p}{|x-y|^{n+sp}} \ dxdy \right)^{\frac{1}{p}}.
    \end{split}
    \end{equation*}
    Since $x \in \Omega$, we note that
    \begin{equation*}
        \operatorname{supp}\{\phi_R(x) -\phi_R(y)\}\subset\{|y| > R\}.
    \end{equation*}
    Therefore,
    \begin{equation*}
        I_R:=\int_{\Omega\times (\mathbb{R}^n \setminus \Omega)}\frac{|\phi_R(x)-\phi_R(y)|^p}{|x-y|^{n+sp}} \ dxdy \le \int_{\mathbb{R}^n\setminus B_R} |1-\phi_R(y)|^p \left(\int_\Omega\frac{dx}{|x-y|^{n+sp}} \ dx\right)\ dy.
    \end{equation*}
    In addition, since $\Omega \subset B_{R_0}$ and $|y|>R$, we have
    \begin{equation*}
        |x-y| \geq |y|-|x| \geq |y| -R_0 > \frac{|y|}{2}.
    \end{equation*}
    Hence,
    \begin{equation*}
        I_R \leq \frac{C(n,s,p)|\Omega|}{R^{sp}} \rightarrow 0 \quad \text{as } R \rightarrow +\infty,
    \end{equation*}
    and the mixed term in~\eqref{weak_pass} vanishes in the limit as $R \rightarrow + \infty$.

    Finally, since $\phi_R \rightarrow 1$ pointwise as $R\rightarrow + \infty$ and since, by Remark~\ref{integr}, $u_\alpha \in L^{q-1}(\mathbb{R}^n \setminus \Omega)$, by the Dominated Convergence Theorem, it follows that
    \begin{equation*}
        \lim_{R \rightarrow + \infty} \int_{\mathbb{R}^n \setminus \Omega} |u_\alpha|^{q-2}u_\alpha \phi_R dx = \int_{\mathbb{R}^n \setminus \Omega}|u_\alpha|^{q-2}u_\alpha dx.
    \end{equation*}
Using these observations and passing to the limit as $R\rightarrow+\infty$ in \eqref{weak_pass}, the thesis follows.
\end{proof}

We now deal with the properties of the map~$\alpha \mapsto E_\alpha$.

\begin{corollary}\label{LIP}
    The map $\alpha \mapsto E_\alpha$ is locally Lipschitz continuous on $(0,+\infty)$.
\end{corollary}

\begin{proof}
This indeed follows from the fact that the quantity $E_\alpha$ can be expressed as the infimum over a family of affine functions in the parameter $\alpha$. It is well
known that the pointwise infimum of affine functions is concave and upper semicontinuous. As a consequence we get the thesis.
\end{proof}

We can actually sharpen Corollary~\ref{LIP} and obtain the following result:

\begin{lemma}\label{COEC1}
    The map $\alpha \mapsto u_\alpha$ is continuous from $(0,+\infty)$ to $\mathcal{W}^{s,p}$. Furthermore, the map $\alpha \mapsto E_\alpha$ is of class $C^1$ and
    \begin{equation*}
        \frac{d E_\alpha}{d\alpha} = \frac{1}{q} \int_{\mathbb{R}^n\setminus \Omega} |u_\alpha|^q \ dx.
    \end{equation*}
\end{lemma}
\begin{proof}
Let $\alpha >0$ and let~$\delta_k $ be an
infinitesimal sequence 
as $k \rightarrow + \infty$, such that $\alpha + \delta_k>0$. Let~$u_{\alpha + \delta_k}$ be the unique minimizer of $J_{\alpha + \delta_k}$, according
to Lemma~\ref{esun},
namely $E_{\alpha + \delta_k} = J_{\alpha+ \delta_k}(u_{\alpha + \delta_k})$. 

We use Corollary~\ref{LIP} to say that, for~$k$ sufficiently large,
$$ |E_{\alpha + \delta_k} |\le |E_{\alpha + \delta_k} - E_{\alpha}|+|E_{\alpha} |\le L\delta_k+ |E_{\alpha} |\le
L+|E_{\alpha} |,$$
where~$L$ is the Lipschitz constant in the interval~$[\alpha,\alpha+1]$.
This entails that, for $k$ sufficiently large, $E_{\alpha+\delta_k}$ is bounded in~$k$. 

Hence, by Lemma~\ref{coercivity}, it follows that $\|u_{\alpha+ \delta_k}\|_{\mathcal{W}^{s,p}} \leq C$, where $C$ is a positive constant independent of~$k$. Therefore, there exists~$v \in \mathcal{W}^{s,p}$ such that, up to a subsequence,
\begin{equation}\label{convst}
u_{\alpha + \delta_k} \rightharpoonup v \quad \text{in } \mathcal{W}^{s,p}\qquad{\mbox{ and}}\qquad
u_{\alpha + \delta_k} \rightarrow v \quad \text{in }L^p(\Omega),
\end{equation}
see e.g.~\cite[Theorem~7.1]{GUIDA} for the compact embedding.

Moreover, since $f \in (\mathcal{W}^{s,p})^*$, by the weak lower semicontinuity of both the seminorm and the $L^q$-norm, together with Corollary~\ref{LIP}, we obtain
\begin{equation*}
    J_\alpha(v) \leq \liminf_{k\rightarrow + \infty} J_{\alpha+\delta_k}(u_{\alpha+\delta_k}) = \lim_{k\rightarrow+\infty} E_{\alpha+\delta_k} = E_\alpha.
\end{equation*}
This says that~$v$ is a minimizer for~$J_\alpha$.

Accordingly, by the uniqueness of the minimizer proved
in Lemma~\ref{esun}, it follows that~$v=u_\alpha$. Therefore, by~\eqref{convst}, we have
\begin{equation}\label{2.11BIS}
u_{\alpha + \delta_k} \rightharpoonup u_\alpha \quad \text{in } \mathcal{W}^{s,p}\qquad{\mbox{and}}\qquad
u_{\alpha + \delta_k} \rightarrow u_\alpha \quad \text{in }L^p(\Omega).
\end{equation}

Now we claim that
\begin{equation}\label{CLAIM}
    u_{\alpha + \delta_k} \rightarrow u_\alpha \quad \text{in } \mathcal{W}^{s,p}.
\end{equation}
Notice that this claim entails the continuity of the map~$\alpha \mapsto u_\alpha$.

In order to prove the claim in~\eqref{CLAIM}, we notice that, by the continuity of~$E_\alpha$ (entailed by Corollary~\ref{LIP}), as~$k \rightarrow + \infty$,
\begin{equation}\label{cndmwiy43764738cghdsvj}
J_{\alpha+\delta_k}(u_{\alpha+\delta_k}) \rightarrow J_\alpha(u_\alpha)\end{equation}
and, since $f\in (\mathcal{W}^{s,p})^*$, it follows from~\eqref{2.11BIS} that,
as~$k \rightarrow+\infty$,
\begin{equation}\label{cndmwiy43764738cghdsvj2}
    \int_\Omega f u_{\alpha+\delta_k}\ dx \rightarrow \int_\Omega fu_\alpha \ dx.
\end{equation}

We now introduce the following quantities:
\begin{eqnarray*}
   && J^1_k : = \frac{c_{n,s,p}}{2p} \int_{\mathbb{R}^{2n}\setminus (\mathbb{R}^n\setminus \Omega)^2} \frac{|u_{\alpha+\delta_k}(x)-u_{\alpha+\delta_k}(y)|^p}{|x-y|^{n+sp}} \ dxdy, \qquad J^2_k :=\frac{\alpha}{q} \int_{\mathbb{R}^n \setminus \Omega} |u_{\alpha+\delta_k}|^q \ dx\\
&&    J^1 : = \frac{c_{n,s,p}}{2p} \int_{\mathbb{R}^{2n}\setminus (\mathbb{R}^n\setminus \Omega)^2} \frac{|u_{\alpha}(x)-u_{\alpha}(y)|^p}{|x-y|^{n+sp}} \ dxdy, \qquad J^2 :=\frac{\alpha}{q} \int_{\mathbb{R}^n \setminus \Omega} |u_{\alpha}|^q \ dx.
\end{eqnarray*}
With this notation, 
exploiting~\eqref{cndmwiy43764738cghdsvj}
and~\eqref{cndmwiy43764738cghdsvj2},
we see that, as~$k \rightarrow + \infty$,
\begin{equation*}
    J^1_k + \left(1+\frac{\delta_k}{\alpha}\right) J^2_k
    =J_{\alpha+\delta_k}(u_{\alpha+\delta_k}) +\int_{\Omega}
    fu_{\alpha+\delta_k}
     \rightarrow 
    J_{\alpha}(u_{\alpha}) +\int_{\Omega}
    fu_{\alpha}
     =     J^1 + J^2.
\end{equation*}
In light of this convergence,
for $k$ large enough, we have the following bounds
\begin{equation}\label{boundJ1J2}
    0 \leq \frac{J^2_k}{2} \leq J^1_k + \left(1+\frac{\delta_k}{\alpha}\right) J^2_k \leq 2(J^1 + J^2).
\end{equation}
As a result,
\begin{equation}\label{limitesomma}
    \lim_{k\rightarrow+\infty} (J_k^1-J^1)+(J_k^2-J^2)
    =\lim_{k\rightarrow+\infty}J^1_k+\left(1+\frac{\delta_k}\alpha\right)J^2_k-
    \frac{\delta_k}\alpha J^2_k-J^1-J^2
    =0.
\end{equation}

We also notice that, by \eqref{boundJ1J2}, $J_k^1$ and $J_k^2$ are uniformly bounded with respect to $k$. Therefore, there exist $\bar J^1$ and $\bar J^2$ such that, up to a subsequence,
\begin{equation*}
    \lim_{k\rightarrow + \infty} J_k^1 = \bar J^1 \qquad \text{and} \qquad \lim_{k\rightarrow + \infty} J_k^2 = \bar J^2.
\end{equation*}
By the weak lower semicontinuity of both the seminorm and the $L^q$-norm, we deduce that
\begin{equation*}
    J^1 \leq \liminf_{k\rightarrow+\infty} J^1_k = \bar J^1 \qquad{\mbox{and}}\qquad J^2 \leq \liminf_{k\rightarrow+\infty} J^2_k = \bar J^2.
\end{equation*}

We claim that
\begin{equation}\label{cms-r473t65i65i}
J^1=\bar J^1\qquad{\mbox{and}}\qquad  J^2=\bar J^2.
\end{equation}
For this, assume, by contradiction, that either $J^1 < \bar J^1$ or $J^2 < \bar J^2$. From \eqref{limitesomma}, we deduce that
\begin{equation*}
    0=\lim_{k\rightarrow+\infty} (J_k^1-J^1)+(J_k^2-J^2) = \bar J^1 - J^1 + \bar J^2 -J^2 > 0.
\end{equation*}
This contradiction establishes~\eqref{cms-r473t65i65i}.

We now observe that~\eqref{cms-r473t65i65i},
together with~\eqref{2.11BIS},
implies~\eqref{CLAIM}.

To compute the first derivative of $E_\alpha$ with respect to $\alpha$, we test the functional $E_{\alpha + \delta_k}$ with $u_\alpha$ and the functional $E_\alpha$ with $u_{\alpha + \delta_k}$, obtaining
    \begin{equation*}
    \begin{split}
        E_\alpha &\leq \frac{c_{n,s,p}}{2p} \int_{\mathbb{R}^{2n}\setminus (\mathbb{R}^n\setminus \Omega)^2} \frac{|u_{\alpha + \delta_k}(x)-u_{\alpha + \delta_k}(y)|^p}{|x-y|^{n+sp}} \ dxdy + \frac{\alpha}{q} \int_{\mathbb{R}^n \setminus \Omega} |u_{\alpha + \delta_k}|^q \ dx - \int_\Omega f u_{\alpha + \delta_k} \ dx\\
        &= E_{\alpha + \delta_k} - \frac{\delta_k}{q} \int_{\mathbb{R}^n \setminus \Omega} |u_{\alpha + \delta_k}|^q \ dx
    \end{split}
    \end{equation*}
    and
    \begin{equation*}
    \begin{split}
        E_{\alpha + \delta_k} &\leq \frac{c_{n,s,p}}{2p} \int_{\mathbb{R}^{2n}\setminus (\mathbb{R}^n\setminus \Omega)^2} \frac{|u_{\alpha}(x)-u_{\alpha}(y)|^p}{|x-y|^{n+sp}} \ dxdy + \frac{\alpha + \delta_k}{q} \int_{\mathbb{R}^n \setminus \Omega} |u_{\alpha}|^q \ dx - \int_\Omega f u_{\alpha} \ dx\\
        &= E_{\alpha} + \frac{\delta_k}{q} \int_{\mathbb{R}^n \setminus \Omega} |u_{\alpha}|^q \ dx.
    \end{split}
    \end{equation*}
    Combining the previous two estimates, we deduce
    \begin{equation*}
        \frac{\delta_k}{q} \int_{\mathbb{R}^n \setminus \Omega} |u_{\alpha + \delta_k}|^q \ dx \leq E_{\alpha + \delta_k} - E_\alpha \leq \frac{\delta_k}{q} \int_{\mathbb{R}^n \setminus \Omega} |u_{\alpha}|^q \ dx.
    \end{equation*}
    By the continuity of the map $\alpha \mapsto u_\alpha$,
    we conclude that
    \begin{equation*}
        \lim_{k \rightarrow + \infty} \frac{E_{\alpha + \delta_k} - E_\alpha}{\delta_k} = \frac{1}{q} \int_{\mathbb{R}^n \setminus \Omega} |u_{\alpha}|^q \ dx,
    \end{equation*} as desired.
\end{proof}

With this preliminary work, we are ready to complete the proof of Theorem~\ref{PF:MA1}.

\begin{proof}[Proof of Theorem~\ref{PF:MA1}] The existence and minimality of~$u_\alpha$
has been established in Lemma~\ref{esun} and
the uniqueness claim is a consequence of Lemma~\ref{LEM:UNI}.

The decay estimate in~\eqref{dePSK:DE} follows from Lemma~\ref{DECAY:TH}
and the integral identity in~\eqref{LINTE}
is given by Lemma~\ref{LINTE:L}.

The regularity properties of the maps~$\alpha
\mapsto u_\alpha$ and~$\alpha \mapsto E_\alpha$, as well as~\eqref{2PSK:-1eidk},
have been proven in Lemma~\ref{COEC1}.\end{proof}

\section{Dirichlet limit and proof of Theorem~\ref{D_exp}}\label{D}

In this section we deal with the asymptotic limit of $E_\alpha$ as $\alpha \rightarrow + \infty$. In this regime, the functional converges to the Dirichlet energy:
\begin{equation*}
    E_\infty = \inf_{u\in W^{s,p}_0(\Omega)} \left\{ \frac{c_{n,s,p}}{2p} \int_{\mathbb{R}^{2n}\setminus (\mathbb{R}^n\setminus \Omega)^2} \frac{|u(x)-u(y)|^p}{|x-y|^{n+sp}} \ dxdy - \int_\Omega f u \ dx\right\},
\end{equation*}which is
achieved by a unique function $u_\infty \in W^{s,p}_0(\Omega)$, where
\begin{equation*}
    W^{s,p}_0(\Omega) := \left\{ u \in L^p(\Omega) : \int_{\mathbb{R}^{2n}\setminus (\mathbb{R}^n\setminus \Omega)^2} \frac{|u(x)-u(y)|^p}{|x-y|^{n+sp}} \ dxdy < + \infty, \ u = 0 \ \text{ in } \mathbb{R}^n \setminus \Omega  \right\}.
\end{equation*}
We now prove the convergence rate of $E_\alpha$ to $E_\infty$, namely the result stated in Theorem~\ref{D_exp}.

\begin{proof}[Proof of Theorem~\ref{D_exp}]
    First of all, we recall that
    \begin{equation*}
        E_\alpha = \inf_{ u \in \mathcal{W}^{s,p}} \left\{\frac{c_{n,s,p}}{2p} \int_{\mathbb{R}^{2n}\setminus (\mathbb{R}^n\setminus \Omega)^2} \frac{|u(x)-u(y)|^p}{|x-y|^{n+sp}} \ dxdy + \frac{\alpha}{q} \int_{\mathbb{R}^n \setminus \Omega} |u|^q \ dx - \int_\Omega f u \ dx\right\}.
    \end{equation*}
    Let $\mu := 1 / (q-1)$. Performing the change of variable $u \mapsto u_\infty + \alpha^{-\mu} u$, we obtain
    \begin{equation*}
    \begin{split}
        E_\alpha &= \frac{c_{n,s,p}}{2p} \int_{\mathbb{R}^{2n}\setminus (\mathbb{R}^n\setminus \Omega)^2} \frac{|u_\infty(x)-u_\infty(y)|^p}{|x-y|^{n+sp}} \ dxdy - \int_\Omega f u_\infty \ dx\\
        &\quad +\inf_{ u \in \mathcal{W}^{s,p}} \Bigg\{ \frac{c_{n,s,p}}{2p}\int_{\mathbb{R}^{2n}\setminus (\mathbb{R}^n\setminus \Omega)^2} \Bigg(\frac{|u_\infty(x)-u_\infty(y) + \alpha^{-\mu}(u(x)-u(y))|^p}{|x-y|^{n+sp}}\\
        &\qquad\qquad\quad\qquad\qquad 
   - \frac{|u_\infty(x)-u_\infty(y)|^p}{|x-y|^{n+sp}}\Bigg)  dxdy + \frac{\alpha^{-\mu}}{q} \int_{\mathbb{R}^n \setminus \Omega} |u|^q dx - \alpha^{-\mu} \int_\Omega fu \ dx \Bigg\},
        \end{split}
        \end{equation*}
        where we used the identity $1-\mu q = -\mu$.
        
        Then, using the integration by parts formula given by
        \begin{equation*}
            \begin{split}
                \int_\Omega fu \ dx &= \int_\Omega (-\Delta)_p^s u_\infty u \ dx\\
                &=  \frac{c_{n,s,p}}{2} \int_{\mathbb{R}^{2n}\setminus (\mathbb{R}^n\setminus \Omega)^2} \frac{|u_\infty(x)-u_\infty(y)|^{p-2}(u_\infty(x)-u_\infty(y))(u(x)-u(y))}{|x-y|^{n+sp}} \ dxdy\\
                &\qquad\ \ - \int_{\mathbb{R}^n \setminus \Omega} N_s^p u_\infty u \ dx,
            \end{split}
        \end{equation*}
        we get
        \begin{equation*}
            E_\alpha = E_\infty + \alpha^{-\mu} R_\alpha,
        \end{equation*}
        where
        \begin{equation*}
        \begin{split}
            R_\alpha &:= \inf_{ u \in \mathcal{W}^{s,p}} \Bigg\{ \alpha^\mu \frac{c_{n,s,p}}{2p}\int_{\mathbb{R}^{2n}\setminus (\mathbb{R}^n\setminus \Omega)^2} \Bigg(\frac{|u_\infty(x)-u_\infty(y) + \alpha^{-\mu}(u(x)-u(y))|^p}{|x-y|^{n+sp}} - \frac{|u_\infty(x)-u_\infty(y)|^p}{|x-y|^{n+sp}}\\
            &\qquad\qquad\qquad\qquad\qquad\quad - p \alpha^{-\mu} \frac{|u_\infty(x)-u_\infty(y)|^{p-2}(u_\infty(x)-u_\infty(y))(u(x)-u(y))}{|x-y|^{n+sp}}\Bigg) dxdy\\
        &\qquad\qquad\qquad\qquad\qquad\quad +  \int_{\mathbb{R}^n \setminus \Omega} \left( \frac{1}{q}|u|^q + N_s^p u_\infty u\right) dx \Bigg\}.
        \end{split}
        \end{equation*}
        
Now, by the convexity of the function $t \mapsto |t|^p$, we obtain the following lower-bound
        \begin{equation}\label{cewr834956gfeuksa}\begin{split}
            R_\alpha &\geq \inf_{ u \in \mathcal{W}^{s,p}} \left\{\int_{\mathbb{R}^n \setminus \Omega} \left( \frac{1}{q}|u|^q + N_s^p u_\infty u\right) dx\right\}\\& \geq \inf_{ u \in L^q(\mathbb{R}^n \setminus \Omega)} \left\{\int_{\mathbb{R}^n \setminus \Omega} \left( \frac{1}{q}|u|^q + N_s^p u_\infty u\right) dx\right\}.\end{split}
        \end{equation}
   
    Suppose now that 
    \begin{equation}\label{nevwuit43yty8oguw}
    |N_s^p u_\infty|^\frac{2-q}{q-1}N_s^p u_\infty \in C_c^\infty(\mathbb{R}^n \setminus \Omega) \end{equation}
    and let $\widetilde u \in C_c^{\infty} (\mathbb{R}^n)$ be an extension to~$\mathbb{R}^n$
    of the function~$- |N_s^p u_\infty|^\frac{2-q}{q-1}N_s^p u_\infty$.

One can check that the infimum in~\eqref{cewr834956gfeuksa} is achieved by the function~$ \widetilde u$. Hence, we obtain the following lower-bound
        \begin{equation*}
        E_\alpha \geq  E_\infty + \frac{1-q}{q} \alpha^{-\mu}\int_{\mathbb{R}^n \setminus \Omega}|N^p_su_\infty|^\frac{q}{q-1} \ dx.
    \end{equation*}
   
    In order to obtain the desired upper-bound, we test the definition of $R_\alpha$ with the function~$\widetilde u$, obtaining
    \begin{equation*}
        \alpha^\mu (E_\alpha - E_\infty) \leq   \frac{1-q}{q} \int_{\mathbb{R}^n \setminus \Omega}|N^p_su_\infty|^\frac{q}{q-1} \ dx + g_\alpha,
    \end{equation*}
    where
    \begin{equation*}
        \begin{split}
           g_\alpha &:= \alpha^\mu\int_{\mathbb{R}^{2n}\setminus (\mathbb{R}^n\setminus \Omega)^2} \Bigg(\frac{|u_\infty(x)-u_\infty(y) + \alpha^{-\mu}(\widetilde u(x)-\widetilde u(y))|^p}{|x-y|^{n+sp}} - \frac{|u_\infty(x)-u_\infty(y)|^p}{|x-y|^{n+sp}}\\
            &\qquad\qquad\qquad\quad - p \alpha^{-\mu} \frac{|u_\infty(x)-u_\infty(y)|^{p-2}(u_\infty(x)-u_\infty(y))(\widetilde u(x)-\widetilde u(y))}{|x-y|^{n+sp}}\Bigg) \ dxdy.
        \end{split}
        \end{equation*}
        
        We claim now that 
        \begin{equation}\label{galpharuvnsbm0987}
        \lim_{\alpha\rightarrow+\infty}g_\alpha =0.
        \end{equation}
         Indeed, if $p>2$, by a Taylor expansion, it follows that
        \begin{equation*}
        \begin{split}
            |g_\alpha| &\leq \alpha^{-\mu} C_p \int_{\mathbb{R}^{2n}\setminus (\mathbb{R}^n\setminus \Omega)^2} \frac{|u_\infty(x)-u_\infty(y)|^{p-2} |\widetilde u(x)- \widetilde u(y)|^2}{|x-y|^{n+sp}} dxdy\\
            &\qquad + \alpha^{(1-p)\mu}C_p \int_{\mathbb{R}^{2n}\setminus (\mathbb{R}^n\setminus \Omega)^2} \frac{ |\widetilde u(x)- \widetilde u(y)|^p}{|x-y|^{n+sp}} dxdy,
        \end{split}
        \end{equation*}
        and, if $p \leq 2$, we have
        \begin{equation*}
            |g_\alpha | \leq \alpha^{(1-p)\mu}C_p \int_{\mathbb{R}^{2n}\setminus (\mathbb{R}^n\setminus \Omega)^2} \frac{ |\widetilde u(x)- \widetilde u(y)|^p}{|x-y|^{n+sp}} dxdy.
        \end{equation*}
        By H\"older's inequality, since $p>1$, the claim in~\eqref{galpharuvnsbm0987} immediately follows.
        
    The remaining part of the proof consists in removing the condition~\eqref{nevwuit43yty8oguw}.
    For this, we use an approximation argument.
    Since by assumption $|N_s^p u_\infty|^\frac{2-q}{q-1}N_s^p u_\infty \in L^q(\mathbb{R}^{n}\setminus \Omega)$, there exists a sequence $\{v_j\} \subset C_c^\infty(\mathbb{R}^{n}\setminus \Omega)$ such that
    \begin{equation}\label{strongL2}
        v_j \rightarrow -|N_s^p u_\infty|^\frac{2-q}{q-1}N_s^p u_\infty \qquad \text{in }\ L^q(\mathbb{R}^{n}\setminus \Omega) \ \text{as} \ j \rightarrow + \infty.
    \end{equation}
    Let $\widetilde v_j \in C_c^\infty(\mathbb{R}^n)$ be an extension of $v_j$ to the whole $\mathbb{R}^n$, namely such that $\widetilde v_j = v_j$ in $\mathbb{R}^{n}\setminus \Omega$. Using the same computations as before, replacing $\widetilde u$ with $\widetilde v_j$, we obtain
    \begin{equation*}
        \begin{split}
            \frac{1-q}{q} \int_{\mathbb{R}^n \setminus \Omega}|N^p_s u_\infty|^\frac{q}{q-1} \ dx &\leq \alpha^\mu (E_\alpha-E_\infty)\\
            &\leq \int_{\mathbb{R}^n \setminus \Omega}\left(\frac{1}{q} |v_j|^q + v_j N^p_su_\infty\right) \ dx + g_{\alpha,j}.
        \end{split}
    \end{equation*}
    Taking the limit as $\alpha \rightarrow +\infty$, since $g_{\alpha,j} \rightarrow 0$ as $\alpha \rightarrow +\infty$, we get
    \begin{equation*}
        \begin{split}
            \frac{1-q}{q} \int_{\mathbb{R}^n \setminus \Omega}|N^p_su_\infty|^\frac{q}{q-1} \ dx &\leq \liminf_{\alpha \rightarrow + \infty}\alpha^\mu (E_\alpha-E_\infty)\\
            &\leq \limsup_{\alpha \rightarrow + \infty}\alpha^\mu (E_\alpha-E_\infty)\\
            &\leq \int_{\mathbb{R}^n \setminus \Omega}\left(\frac{1}{q} |v_j|^q + v_j N^p_su_\infty \right)\ dx.
        \end{split}
    \end{equation*}
 
 We now take the limit as $j \rightarrow + \infty$ and, using \eqref{strongL2}, we conclude that
    \begin{equation*}
        \lim_{\alpha \rightarrow + \infty} \alpha^\mu(E_\alpha - E_\infty) = \frac{1-q}{q} \int_{\mathbb{R}^n \setminus \Omega}|N^p_su_\infty|^\frac{q}{q-1} \ dx.
    \end{equation*}
    
Gathering all the pieces of information, we have thereby established that, as~$\alpha \rightarrow + \infty$,
    \begin{equation*}
        E_\alpha = E_\infty + \alpha^{-\mu}\frac{1-q}{q} \int_{\mathbb{R}^n \setminus \Omega}|N^p_su_\infty|^\frac{q}{q-1} \ dx + o(\alpha^{-\mu}),
    \end{equation*} as desired.
\end{proof}

\section{Neumann limit and proof of Theorems~\ref{N_main} and~\ref{int_f_div}.}\label{N}
In this section, we study the asymptotic behavior of the functional $E_\alpha$ as $\alpha \to 0^+$, distinguishing the case in which it converges to the energy of a solution of the nonlocal Neumann problem from the case in which no such limit corresponds to a solution.
\subsection{The case \texorpdfstring{$\int_\Omega fdx=0$}{int f = 0}}
Let us consider the minimization problem
\begin{equation*}
    E_0 = \inf_{u \in W^{s,p}(\Omega)} J_0(u),
\end{equation*}
where
$$ J_0(u):=\frac{c_{n,s,p}}{2p} \int_{\mathbb{R}^{2n}\setminus (\mathbb{R}^n\setminus \Omega)^2} \frac{|u(x)-u(y)|^p}{|x-y|^{n+sp}} \ dxdy - \int_\Omega f u \ dx.$$
Notice that $E_0$ is achieved by a function $u_0 \in W^{s,p}(\Omega)$, unique up to additive constants, which exists in view of the compatibility condition $\int_\Omega f dx=0$. Notice that in this framework we consider $f \in (W^{s,p}(\Omega))^*$.

Moreover, any minimizer $u_0 \in W^{s,p}(\Omega)$ weakly solves
\begin{equation}\label{N_n_l}
    \begin{cases}
        (-\Delta)_p^s u_0 = f & \text{in } \Omega,\\
        N^p_s u_0 =0 & \text{in } \mathbb{R}^n \setminus \overline \Omega.
    \end{cases}
\end{equation}
%
%
%

We point out that
the convergence of the Robin functional~$E_\alpha$ to the Neumann functional~$E_0$ as~$\alpha \to 0^+$ is not straightforward. By~\cite[Proposition~3.13]{DRV}, any weak solution $u_0$ of the nonlocal Neumann problem \eqref{N_n_l}, with $p=2$, satisfies
    \begin{equation*}
        \lim_{|x|\rightarrow + \infty} u_0(x) = \frac{1}{|\Omega|} \int_\Omega u_0 \ dx, \quad \text{uniformly in } x.
    \end{equation*}
    Hence, also in the linear case, any minimizer $u_0$ achieving $E_0$ does not, in general, belong to~$L^q(\mathbb{R}^n \setminus \Omega)$, so we cannot directly test the definition of $E_\alpha$ using $u_0 \in W^{s,p}(\Omega)$. 
    
    To overcome this difficulty, we consider a cut-off function~$\eta_R \in C_c^\infty(B_{2R},[0,1])$, with~$\eta_R=1 $, in~$B_R$
    and we define
    a modified function 
    \begin{equation}\label{deffunctiouzero000}
    u_{0,R}:=T_R(u_0) \eta_R,\end{equation}
    where~$T_R(u_0):=\min\{u_0,R\}$
    
       As a consequence, for any given $R>0$, the function $u_{0,R}$ belongs to $L^q(\mathbb{R}^n \setminus \Omega)$, for any~$q > 1$. We then prove the following result, which establishes the convergence of the functional $J_0$ evaluated at $u_{0,R}$ to the minimum $E_0$ as $R \rightarrow + \infty$.

\begin{lemma}\label{convCO}
    Let $u_0 \in W^{s,p}(\Omega)$ be a minimizer achieving $E_0$.
    Then,
    \begin{equation*}
      \lim_{R\to+\infty}  J_0(u_{0,R})= E_0.
    \end{equation*}
\end{lemma}

\begin{proof}
We observe that~$u_{0,R}\in W^{s,p}(\Omega)$ and we
set~$\omega_R := u_{0,R} - u_0$. Notice that $\omega_R$ can be rewritten as
    \begin{eqnarray*}
        &&\omega_R = \omega_{R,1} + \omega_{R,2},\\
        &&{\mbox{where }}\quad
         \omega_{R,1}:=
        (T_R(u_0)-u_0)\eta_R \quad{\mbox{and}}\quad
         \omega_{R,2}:= u_0(\eta_R - 1)
  .  \end{eqnarray*}
   
    We set
    $$    \widetilde J_R:=\int_{\mathbb{R}^{2n}\setminus (\mathbb{R}^n\setminus \Omega)^2} \frac{|\omega_R(x)-\omega_R(y)|^p}{|x-y|^{n+sp}} \ dxdy$$ and
we claim that
    \begin{equation}\label{J_tilde}
       \lim_{R\to+\infty} \widetilde J_R=0.
    \end{equation}
    First of all, by the convexity of the map $t \mapsto |t|^p$, we deduce that
    \begin{equation}\label{J_tildeBIS}
        |\omega_R(x)-\omega_R(y)|^p \leq 2^{p-1} \Big(|\omega_{R,1}(x)-\omega_{R,1}(y)|^p + |\omega_{R,2}(x)-\omega_{R,2}(y)|^p\Big).
    \end{equation}

   Let us now deal with the term $\omega_{R,1}$. By convexity, we obtain
    \begin{equation}\label{y84yghjabvja-fejwjkf6543frgehdju435irufw0}
        |\omega_{R,1}(x)-\omega_{R,1}(y)|^p \leq 2^{p-1}\Big(|\xi_R(x)-\xi_R(y)|^p |\eta_R(x)|^p + |\xi_R(y)|^p |\eta_R(x)-\eta_R(y)|^p\Big),
    \end{equation}
    where $\xi_R := T_R(u_0)-u_0$. 

We know that $\xi_R(x)\rightarrow 0$ a.e. in $\mathbb{R}^{n}$ as $R \rightarrow + \infty$, and moreover
    \begin{equation*}
        \frac{|\xi_R(x)-\xi_R(y)|^p\eta_R^p(x)}{|x-y|^{n+sp}} \leq  \frac{2^p|u_0(x)-u_0(y)|^p}{|x-y|^{n+sp}} \in L^1(\mathbb{R}^{2n}\setminus (\mathbb{R}^n\setminus \Omega)^2).
    \end{equation*}
Thus, by the Dominated Convergence Theorem, it follows that
    \begin{equation}\label{y84yghjabvja-fejwjkf6543frgehdju435irufw}
     \lim_{R\to+\infty}   \int_{\mathbb{R}^{2n}\setminus (\mathbb{R}^n\setminus \Omega)^2} \frac{{|\xi_R(x)-\xi_R(y)|}^p\eta_R^p(x)}{|x-y|^{n+sp}} \ dxdy= 0.
    \end{equation}
    
Furthermore, since $\Omega$ is bounded, there exists $R_0>0$ such that $\Omega \subset B_{R_0}$. Therefore, we can choose $R>2R_0$ in such a way that 
\begin{equation}\label{cdst4893ytfklhjYUdh3}
{\mbox{$\eta_R(x)=1$ for any $x\in \Omega$.}}\end{equation}
    
    On this account,
    \begin{equation*}\begin{split}
        &\int_{\mathbb{R}^{2n}\setminus (\mathbb{R}^n\setminus \Omega)^2} \frac{|\xi_R(y)|^p |\eta_R(x)-\eta_R(y)|^p}{|x-y|^{n+sp}} \ dxdy 
        \\&=\int_{\Omega\times (\mathbb{R}^n\setminus B_R)} \frac{|\xi_R(y)|^p |1-\eta_R(y)|^p}{|x-y|^{n+sp}} \ dxdy 
        +\int_{(\mathbb{R}^n\setminus B_R)\times\Omega} \frac{|\xi_R(y)|^p |\eta_R(x)-1|^p}{|x-y|^{n+sp}} \ dxdy .
            \end{split}
    \end{equation*}
    We also observe that
    $$ |\xi_R(y)|^p= |\xi_R(y)-\xi_R(x)+\xi_R(x)|^p\leq
    2^p\Big(|\xi_R(y)-\xi_R(x)|^p+|\xi_R(x)|^p\Big)$$
    and accordingly
    \begin{equation}\label{cnr3727fhawDFGHdiewuifhl8760}
    \begin{split}
        &\int_{\mathbb{R}^{2n}\setminus (\mathbb{R}^n\setminus \Omega)^2} \frac{|\xi_R(y)|^p |\eta_R(x)-\eta_R(y)|^p}{|x-y|^{n+sp}} \ dxdy 
        \\&\leq 2^p\int_{\Omega\times (\mathbb{R}^n\setminus B_R)} \frac{|\xi_R(x)-\xi_R(y)|^p |1-\eta_R(y)|^p}{|x-y|^{n+sp}} \ dxdy 
        \\&\qquad+(2^p+1)\int_{(\mathbb{R}^n\setminus B_R)\times\Omega} \frac{|\xi_R(y)|^p |\eta_R(x)-1|^p}{|x-y|^{n+sp}} \ dxdy .
            \end{split}
    \end{equation}
    
Since~$\xi_R(x)\rightarrow 0$ a.e. in~$\mathbb{R}^{n}$ as~$R \rightarrow + \infty$ and
    \begin{equation*}
    \frac{|\xi_R(x)-\xi_R(y)|^p |1-\eta_R(y)|^p}{|x-y|^{n+sp}}\leq
  \frac{2^p|u_0(x)-u_0(y)|^p}{|x-y|^{n+sp}} \in L^1(\Omega\times(\mathbb{R}^{n}\setminus B_R)),
    \end{equation*}
we can employ the Dominated Convergence Theorem and obtain that
   \begin{equation*}
    \lim_{R\to+\infty}   \int_{\Omega\times (\mathbb{R}^n\setminus B_R)} \frac{|\xi_R(x)-\xi_R(y)|^p |1-\eta_R(y)|^p}{|x-y|^{n+sp}} \ dxdy = 0.
    \end{equation*}
   From this and~\eqref {cnr3727fhawDFGHdiewuifhl8760} we conclude that
  \begin{equation}\label{cnr3727fhawDFGHdiewuifhl876}
    \begin{split}
        &\lim_{R\to+\infty}\int_{\mathbb{R}^{2n}\setminus (\mathbb{R}^n\setminus \Omega)^2} \frac{|\xi_R(y)|^p |\eta_R(x)-\eta_R(y)|^p}{|x-y|^{n+sp}} \ dxdy 
        \\&\qquad\leq (2^p+1)\lim_{R\to+\infty}\int_{(\mathbb{R}^n\setminus B_R)\times\Omega} \frac{|\xi_R(y)|^p |\eta_R(x)-1|^p}{|x-y|^{n+sp}} \ dxdy .
            \end{split}
    \end{equation}  

Moreover, if~$x\in\mathbb{R}^n\setminus B_R$ and~$y\in \Omega$, then
$$|x-y| \geq |x|-|y| \geq \frac{|x|}{2}.$$
As a consequence,
\begin{eqnarray*}
\int_{(\mathbb{R}^n\setminus B_R)\times\Omega} \frac{|\xi_R(y)|^p |\eta_R(x)-1|^p}{|x-y|^{n+sp}} \ dxdy
&\leq&
2^{n+sp}\int_{\mathbb{R}^n\setminus B_R}\frac{dx}{|x|^{n+sp}}\int_\Omega |\xi_R(y)|^p \ dy\\
&\leq&\frac{C(n,s,p) \|u_0\|_{L^p(\Omega)}^p}{R^{sp}}.
\end{eqnarray*}
Plugging this into~\eqref{cnr3727fhawDFGHdiewuifhl876}, we obtain that
\begin{equation*}\lim_{R\to+\infty}
\int_{\mathbb{R}^{2n}\setminus (\mathbb{R}^n\setminus \Omega)^2} \frac{|\xi_R(y)|^p |\eta_R(x)-\eta_R(y)|^p}{|x-y|^{n+sp}} \ dxdy 
\leq \lim_{R\to+\infty} \frac{C(n,s,p) \|u_0\|_{L^p(\Omega)}^p}{R^{sp}}=0.
    \end{equation*}
 
 This, \eqref{y84yghjabvja-fejwjkf6543frgehdju435irufw0}  
    and~\eqref{y84yghjabvja-fejwjkf6543frgehdju435irufw}
    give that
    \begin{equation}\label{J_tildeTER}
    \lim_{R\to+\infty} \int_{\mathbb{R}^{2n}\setminus (\mathbb{R}^n\setminus \Omega)^2} \frac{|\omega_{R,1}(x)-\omega_{R,1}(y)|^p}{|x-y|^{n+sp}} \ dxdy=0.
    \end{equation}
    
    We now turn our attention to the term $\omega_{R,2}$. First of all, by the definition of $\omega_{R,2}$, it follows that $\operatorname{supp}\{\omega_{R,2}\} \subset \mathbb{R}^n\setminus B_R$. Hence, recalling~\eqref{cdst4893ytfklhjYUdh3},
    \begin{equation}\label{ceriwfgukewti43u9876098760}
   \int_{\mathbb{R}^{2n}\setminus (\mathbb{R}^n\setminus \Omega)^2} \frac{|\omega_{R,2}(x)-\omega_{R,2}(y)|^p}{|x-y|^{n+sp}} \ dxdy = 2  \,\widetilde I_R
    \end{equation}
    where
    \begin{equation*}  \widetilde I_R:=
      \int_{(\mathbb{R}^n\setminus B_R)\times\Omega } \frac{|\omega_{R,2}(x)-\omega_{R,2}(y)|^p}{|x-y|^{n+sp}} \ dxdy .
    \end{equation*}

    We now focus on proving that 
    \begin{equation}\label{ceriwfgukewti43u987609876}\lim_{R\to+\infty}\widetilde I_R = 0
    .\end{equation} By the definition of $\omega_{R,2}$ and by convexity, we have
    \begin{equation*}
    \begin{split}
        |\omega_{R,2}(x) - \omega_{R,2}(y)|^p &= |u_0(x)(\eta_R(x)-1) - u_0(y)(\eta_R(y)-1)|^p\\
        &\leq 2^{p-1}\Big(|u_0(x)-u_0(y)|^p|\eta_R(x)-1|^p + |u_0(y)|^p|\eta_R(x)-\eta_R(y)|^p\Big).
    \end{split}
    \end{equation*}
    Therefore, 
    \begin{equation*}
    \begin{split}
        \widetilde I_R &\leq 2^{p} \int_{(\mathbb{R}^n\setminus B_R)\times\Omega } \frac{|u_0(x)-u_0(y)|^p|\eta_R(x)-1|^p}{|x-y|^{n+sp}} \ dxdy\\
        &\qquad + 2^{p}\int_{(\mathbb{R}^n\setminus B_R)\times\Omega }
         |u_0(y)|^p \frac{|\eta_R(x)-\eta_R(y)|^p}{|x-y|^{n+sp}} \ dxdy
         .
    \end{split}
    \end{equation*}
    
    Since $\eta_R \rightarrow 1$ as $R\rightarrow +\infty$ and
        \begin{equation*}
        \frac{|u_0(x)-u_0(y)|^p|\eta_R(x)-1|^p}{|x-y|^{n+sp}} \leq \frac{|u_0(x)-u_0(y)|^p}{|x-y|^{n+sp}} \in L^1((\mathbb{R}^n \setminus  B_R) \times \Omega),
    \end{equation*}
by the Dominated Convergence Theorem we deduce that 
\begin{equation*}
\lim_{R\to+\infty}
\int_{(\mathbb{R}^n\setminus B_R)\times\Omega } \frac{|u_0(x)-u_0(y)|^p|\eta_R(x)-1|^p}{|x-y|^{n+sp}} \ dxdy=0.
\end{equation*}

Furthermore,    \begin{eqnarray*}
\int_{(\mathbb{R}^n\setminus B_R)\times\Omega }
         |u_0(y)|^p \frac{|\eta_R(x)-\eta_R(y)|^p}{|x-y|^{n+sp}} \ dxdy &\leq& \int_\Omega |u_0(y)|^p\ dy
         \int_{\mathbb{R}^n\setminus B_R} \frac{2^{n+sp}\,dx}{|x|^{n+sp}} \\& \leq&\frac{ C(n,s,p) \|u_0\|_{L^p(\Omega)} }{R^{sp}}.
    \end{eqnarray*}
    Gathering these pieces of information, we deduce~\eqref{ceriwfgukewti43u987609876}.
    
 From~\eqref{ceriwfgukewti43u9876098760}
 and~\eqref{ceriwfgukewti43u987609876}, we thus conclude that
 $$ \lim_{R\to+\infty}\int_{\mathbb{R}^{2n}\setminus (\mathbb{R}^n\setminus \Omega)^2} \frac{|\omega_{R,2}(x)-\omega_{R,2}(y)|^p}{|x-y|^{n+sp}} \ dxdy 
    =0.$$
    This, \eqref{J_tildeBIS} and~\eqref{J_tildeTER} entail~\eqref{J_tilde}.

    Moreover, by the definition of $\omega_R$ and by Taylor's expansion, it follows that
    \begin{equation*}
    \begin{split}
        &\int_{\mathbb{R}^{2n}\setminus (\mathbb{R}^n\setminus \Omega)^2} \frac{|u_{0,R}(x)-u_{0,R}(y)|^p}{|x-y|^{n+sp}} \ dxdy\\
        &\quad= \int_{\mathbb{R}^{2n}\setminus (\mathbb{R}^n\setminus \Omega)^2} \frac{|u_0(x)-u_0(y)|^p}{|x-y|^{n+sp}} \ dxdy\\
        &\qquad\qquad+p\int_{\mathbb{R}^{2n}\setminus (\mathbb{R}^n\setminus \Omega)^2} \frac{|u_0(x)-u_0(y)|^{p-2}(u_0(x)-u_0(y))(\omega_R(x)-\omega_R(y))}{|x-y|^{n+sp}} \ dxdy\\
        &\qquad\qquad + \int_{\mathbb{R}^{2n}\setminus (\mathbb{R}^n\setminus \Omega)^2} \frac{R(u_0,\omega_R)}{|x-y|^{n+sp}} \ dxdy\\
        &\quad =: I_{R,1}+I_{R,2}+I_{R,3},
    \end{split}
    \end{equation*}
    where the reminder $R(u_0,\omega_R)$ satisfies
    \begin{equation*}
        |R(u_0,\omega_R)|\leq \begin{cases}
        C_p \left(|u_0(x)-u_0(y)|^{p-2}|\omega_R(x)- \omega_R(y)|^2+ |\omega_R(x)- \omega_R(y)|^p \right), \quad & \text{ if } \ p >2,
    \\ C_p |\omega_R(x)- \omega_R(y)|^p, \quad & \text{ if } \ p \leq 2.\end{cases}
    \end{equation*}
    By \eqref{J_tilde}, and by H\"older's inequality with exponents $(p/p-1,p)$, it follows that
    \begin{equation*}
        |I_{R,2}| \leq p I_1^{\frac{p-1}{p}} \widetilde J_R^{\frac{1}{p}} \rightarrow 0 \quad \text{ as } R \rightarrow + \infty.
    \end{equation*}
    In addition, if $p > 2$, by H\"older's inequality with exponents $(p/p-2,p/2)$, we have
    \begin{equation*}
        |I_{R,3}| \leq C_p\left( I_1^{\frac{p-2}{p}}\widetilde{J}^{\frac{2}{p}} + \widetilde J_R\right) \rightarrow 0 \quad \text{ as } R \rightarrow + \infty.
    \end{equation*}
    If instead $p \leq 2$, the following estimate holds
    \begin{equation*}
        |I_{R,3}| \leq C_p \widetilde J_R \rightarrow 0 \quad \text{ as } R \rightarrow + \infty.
    \end{equation*}
    Hence, we conclude that
    \begin{equation*}
        \lim_{R \rightarrow + \infty} \int_{\mathbb{R}^{2n}\setminus (\mathbb{R}^n\setminus \Omega)^2} \frac{|u_{0,R}(x)-u_{0,R}(y)|^p}{|x-y|^{n+sp}} \ dxdy = \int_{\mathbb{R}^{2n}\setminus (\mathbb{R}^n\setminus \Omega)^2} \frac{|u_{0}(x)-u_{0}(y)|^p}{|x-y|^{n+sp}} \ dxdy.
    \end{equation*}
    
    Furthermore, by the Dominated Convergence Theorem, we have
    \begin{equation}\label{4830jdshmDFGHJui3yri3}
        \lim_{R\rightarrow + \infty}\int_\Omega  |u_{0,R}|^p dx = \int_\Omega |u_0|^p dx.
    \end{equation}
    As a consequence, we deduce that $u_{0,R} \rightarrow u_0$ in $W^{s,p}(\Omega)$ as $R \rightarrow + \infty$. In particular, since~$f \in (W^{s,p}(\Omega))^*$,
    \begin{equation}\label{4830jdshmDFGHJui3yri32}
      \lim_{R\rightarrow + \infty}  \int_\Omega f u_{0,R} dx =\int_\Omega f u_0 dx .
    \end{equation}
 The desired convergence
 result then follows from~\eqref{4830jdshmDFGHJui3yri3}
 and~\eqref{4830jdshmDFGHJui3yri32}.
\end{proof}

We are now in the position to prove Theorem~\ref{N_main}.

\begin{proof}[Proof of Theorem~\ref{N_main}]
   Testing the definition of $E_\alpha$ with the function $u_{0,R} \in \mathcal{W}^{s,p}$, defined in~\eqref{deffunctiouzero000}, we obtain the following upper bound
    \begin{equation*}
    \begin{split}
        E_\alpha &\leq \frac{c_{n,s,p}}{2p} \int_{\mathbb{R}^{2n}\setminus (\mathbb{R}^n\setminus \Omega)^2} \frac{|u_{0,R}(x)-u_{0,R}(y)|^p}{|x-y|^{n+sp}} \ dxdy + \frac{\alpha}{q} \int_{\mathbb{R}^n \setminus \Omega} |u_{0,R}|^q \ dx - \int_\Omega f u_{0,R} \ dx\\
        &= J_0(u_{0,R})  + \frac{\alpha}{q} \int_{\mathbb{R}^n \setminus \Omega} |u_{0,R}|^q \ dx.
    \end{split}
    \end{equation*}
    Moreover, for any $\alpha > 0$, we know that $E_\alpha \geq J_0(u_\alpha)$. Therefore,
    \begin{equation}\label{cnewytu4583ui-7654tr3e}
        E_0 \leq E_\alpha \leq J_0(u_{0,R})  + \frac{\alpha}{q} \int_{\mathbb{R}^n \setminus \Omega} |u_{0,R}|^q \ dx.
    \end{equation}
    
    We observe that~$u_{0,R}\in L^q(\mathbb{R}^n \setminus \Omega)$ for any~$R>0$. As a result,
passing to the limit as $\alpha \rightarrow 0^+$ in~\eqref{cnewytu4583ui-7654tr3e}, we obtain
    \begin{equation*}
        E_0 \leq \liminf_{\alpha \rightarrow 0^+} E_\alpha \leq \limsup_{\alpha \rightarrow 0^+} E_\alpha \leq J_0(u_{0,R}).
    \end{equation*}
    Finally, taking the limit as $R \rightarrow + \infty$ and exploiting Lemma~\ref{convCO}, it follows that
    \begin{equation*}
        \lim_{\alpha \rightarrow 0^+} E_\alpha = E_0,
    \end{equation*} as desired.
\end{proof}
\subsection{The case \texorpdfstring{$\int_\Omega fdx \neq 0$}{int f =/ 0}}
In this case, the Neumann problem does not admit a solution, and the behavior of $E_\alpha$ as $\alpha \rightarrow 0^+$ turns out to be completely different. Here we provide the details of
the proof of Theorem~\ref{int_f_div}.

\begin{proof}[Proof of Theorem~\ref{int_f_div}]
    Let~$\eta_R \in C_c^\infty(B_{2R},[0,1])$, with~$\eta_R=1 $, in~$B_R$. Notice that, since $\Omega$ is bounded, there exists $R_0>0$ such that $\Omega \subset B_{R_0}$. Thus, we choose $R>2R_0$, in such a way that~$\eta_R=1$ in~$\Omega$.
    
    Let also $v_t(x) := t \eta_R(x)$, for $x \in \mathbb{R}^n$ and $t\in \mathbb{R}$, and set
    \begin{eqnarray*}
    \beta_R&:=&\frac{c_{n,s,p}}{2p} \int_{\mathbb{R}^{2n}\setminus (\mathbb{R}^n\setminus \Omega)^2} \frac{|\eta_R(x)-\eta_R(y)|^p}{|x-y|^{n+sp}} \ dxdy,\\
   \gamma_R&:=& \int_{\mathbb{R}^n \setminus \Omega} \eta_R^q(x) \ dx ,\\
   {\mbox{and }}\quad F&:=& \int_\Omega f\ dx.
    \end{eqnarray*} 
      It follows that
    \begin{eqnarray*}&&
        \frac{c_{n,s,p}}{2p} \int_{\mathbb{R}^{2n}\setminus (\mathbb{R}^n\setminus \Omega)^2} \frac{|v_t(x)-v_t(y)|^p}{|x-y|^{n+sp}} \ dxdy =|t|^p \beta_R,\\
&&        \frac{\alpha}{q} \int_{\mathbb{R}^n \setminus \Omega} |v_t(x)|^q dx = \frac{\alpha}{q} |t|^q \gamma_R,\\
{\mbox{and }} &&        \int_\Omega fv_t \ dx = tF.
    \end{eqnarray*}
    Thus, collecting the previous identities, we obtain
    \begin{equation}\label{min_on_t}
        J_\alpha(v_t) = |t|^p \beta_R + |t|^q\frac{\alpha}{q}\gamma_R - tF.
    \end{equation}
    
    We start considering the case $p=q$. A direct computation shows that $J_\alpha(v_t)$ is minimized at
    \begin{equation*}
        t_\alpha: = \frac{\operatorname{sign}{F}}{q^{\frac{1}{q-1}}}\left(\frac{|F|}{\beta_R + \frac{\alpha}{q} \gamma_{R}}\right)^{\frac{1}{q-1}}.
    \end{equation*}
    Hence, testing the definition of $E_\alpha$ with the function $v_{t_\alpha}$, we get
    \begin{equation}\label{stimainf}
        E_\alpha \leq J_\alpha(v_{t_\alpha}) =-\frac{q-1}{q} |F| \left(\frac{|F|}{q\beta_R + \alpha \gamma_R}\right)^{\frac{1}{q-1}}.
    \end{equation}
   
    Let us now prove an upper bound for $J_\alpha(v_{t_\alpha})$. By the properties of the cutoff function $\eta_R$, we have
    \begin{equation}\label{bound1}
        \gamma_{R}= \int_{\mathbb{R}^n \setminus \Omega} \eta_R^q (x)\ dx = \int_{B_{2R} \setminus \Omega} \eta_R^q(x) \ dx \leq CR^n,
    \end{equation}
    and
     \begin{equation}\label{bound2}
    \begin{split}
      &  \beta_R =  \frac{c_{n,s,p}}{2p} \int_{\mathbb{R}^{2n}\setminus (\mathbb{R}^n\setminus \Omega)^2} \frac{|\eta_R(x)-\eta_R(y)|^p}{|x-y|^{n+sp}} \ dxdy = \frac{c_{n,s,p}}{p} \int_{\Omega\times (\mathbb{R}^{n}\setminus \Omega)} \frac{|1-\eta_R(y)|^p}{|x-y|^{n+sp}} \ dxdy
    \\&\quad\leq \frac{c_{n,s,p}}{p} \int_{\Omega\times 
    (\mathbb{R}^{n}\setminus B_R)} \frac{dx dy}{|x-y|^{n+sp}}  \leq \frac{2^{n+sp}c_{n,s,p}|\Omega| }{p} \int_{\mathbb{R}^n\setminus B_R} \frac{dy}{|y|^{n+sp}} \\
        &\quad\leq \frac{2^{n+sp}c_{n,s,p} |\Omega|\,|\partial B_1|}{sp^2 R^{sp}}.
    \end{split}
    \end{equation}
Hence, we obtain the following bound
    \begin{equation}\label{boundsopra}
        q \beta_R +  \alpha \gamma_{R} \leq \frac{C_1}{ R^{sp}} + C_2 \alpha R^n,
    \end{equation}
    where $C_1 = C_1(n,s,p,q,|\Omega|)$ and $C_2=C_2(n)$ are positive constants.
    
    Let us now choose~$
        R:= \alpha ^{- \frac{1}{n+sp}}$.
    Thus, using the estimate \eqref{boundsopra}, the inequality \eqref{stimainf} becomes
    \begin{equation*}
        E_\alpha \leq - \overline{ C} |F|^{\frac{q}{q-1}}\alpha^{-\frac{sp}{(n+sp)(q-1)}},
    \end{equation*}
    where $\overline{ C}= \overline{ C}(n,s,p,q,|\Omega|)$ is a positive constant.
    
    Finally, passing to the limit as $\alpha \rightarrow 0^+$, we conclude that
    \begin{equation*}
        \lim_{\alpha \rightarrow 0^+} E_\alpha = -\infty.
    \end{equation*}
    
    Let us now deal with the case $p \neq q$. The study of the behavior of the functional $E_\alpha$ becomes more involved, since the minimum of the expression in \eqref{min_on_t} with respect to $t$ cannot be determined explicitly. To overcome this difficulty, we make a suitable choice of $t$, which will become clear later, namely,
    \begin{equation*}
        t = t_m := \alpha^{\frac{1+\beta(n+sp)}{p-q}}\operatorname{sign}{F},
    \end{equation*}
    where
    \begin{equation}\label{PD}
        \beta \in \left(\min\left\{-\frac{1}{n+sp},-\frac{1}{n+sq}\right\}, \max\left\{-\frac{1}{n+sp},-\frac{1}{n+sq}\right\}\right).
    \end{equation}
    With this choice, and using the bounds \eqref{bound1} and \eqref{bound2}, we deduce that
    \begin{equation}\label{PD1}
    \begin{split}
        J_\alpha(v_{t_m}) &= \alpha^{\frac{p(1+\beta(n+sp))}{p-q}}{|\operatorname{sign}{F}|}^p \beta_R + \alpha^{\frac{q(1+\beta(n+sp))}{p-q}}|\operatorname{sign}{F}|^q \frac{\alpha}{q} \gamma_R - \alpha^{\frac{1+\beta(n+sp)}{p-q}}|F|\\
        &\leq C_1 \alpha^{\frac{p(1+\beta(n+sp))}{p-q}}{|\operatorname{sign}{F}|}^p R^{-sp} + C_2\alpha^{\frac{q(1+\beta(n+sp))}{p-q}}|\operatorname{sign}{F}|^q \frac{\alpha}{q} R^n - \alpha^{\frac{1+\beta(n+sp)}{p-q}}|F|.
    \end{split}
    \end{equation}
We choose here~$
        R := \alpha^\beta$,
    with $\beta < 0$ as in~\eqref{PD}. Therefore, the estimate in~\eqref{PD1} becomes
    \begin{equation*}
        J_\alpha(v_{t_m}) \leq C_1 \alpha^{\frac{p(1+\beta(n+sp))}{p-q}-sp\beta}|\operatorname{sign}{F}|^p + C_2\alpha^{\frac{q(1+\beta(n+sp))}{p-q}+1+\beta n}|\operatorname{sign}{F}|^q - \alpha^{\frac{1+\beta(n+sp)}{p-q}}|F|.
    \end{equation*}
    With this particular choice of the parameter $\beta$, it is straightforward to verify that
    \begin{equation*}
        \lim_{\alpha \rightarrow 0^+} \alpha^{\frac{p(1+\beta(n+sp))}{p-q}-sp\beta} = \lim_{\alpha\rightarrow 0^+} \alpha^{\frac{q(1+\beta(n+sp))}{p-q}+1+\beta n} =0
    \end{equation*}
    and
    \begin{equation*}
        \lim_{\alpha \rightarrow 0^+} \alpha^{\frac{1+\beta(n+sp)}{p-q}} = + \infty.
    \end{equation*}
    Therefore, the proof concludes by taking the limit as $\alpha \to 0^+$.
\end{proof}

\begin{remark}
    The use of the cutoff function $\eta_R \in C_c^\infty(\mathbb{R}^n)$ in the proof of Theorem~\ref{int_f_div} is necessary because constant functions do not belong to the space $\mathcal{W}^{s,p}$. Therefore, they cannot be used directly to test the definition of the energy $E_\alpha$.
\end{remark}
\appendix
\section{}\label{A}

\begin{proposition}\label{minimo_weak}
    Let $\alpha > 0$ and $f\in (\mathcal{W}^{s,p})^*$. Let~$J_\alpha : \mathcal{W}^{s,p} \rightarrow \mathbb{R}$ be the functional defined in \eqref{FUNCTIONAL}.
    
    Then, any critical point of $J_\alpha$ is a weak solution of \eqref{pb:RF}.
\end{proposition}

\begin{proof}
    Firstly, since $u \in \mathcal{W}^{s,p}$ and $f \in (\mathcal{W}^{s,p})^*$, the functional $J_\alpha$ is well defined for any $\alpha > 0$. \\
 Now, we compute the first variation of $J_\alpha$. Let $|\varepsilon|<1$ and $\varphi \in \mathcal{W}^{s,p}$, then $u+\varepsilon \varphi \in \mathcal{W}^{s,p}$. Therefore, we are now in the position to compute 
    \begin{equation*}
    \begin{split}
        J_\alpha(u+\varepsilon v) &= \frac{c_{n,s,p}}{2p} \int_{\mathbb{R}^{2n}\setminus (\mathbb{R}^n\setminus \Omega)^2} \frac{|(u + \varepsilon \varphi)(x)-(u + \varepsilon \varphi)(y)|^p}{|x-y|^{n+sp}} \ dxdy + \frac{\alpha}{q} \int_{\mathbb{R}^n \setminus \Omega} |u + \varepsilon \varphi|^q \ dx\\
        &\qquad\qquad- \int_\Omega f (u + \varepsilon \varphi) \ dx.
    \end{split}
    \end{equation*}
    
    Moreover, by a Taylor expansion, we obtain
    \begin{equation*}
    \begin{split}
        &|(u + \varepsilon \varphi)(x)-(u + \varepsilon \varphi)(y)|^p\\
        &\quad\quad= |u(x)-u(y)|^p + \varepsilon p |u(x)-u(y)|^{p-2}(u(x)-u(y))(\varphi(x)-\varphi(y)) +  R^\varepsilon_p(u,\varphi),
    \end{split}
    \end{equation*}
    and
    \begin{equation*}
        |u+\varepsilon \varphi|^q = |u|^q + \varepsilon q |u|^{q-2}u\varphi +  R^\varepsilon_q(u,\varphi),
    \end{equation*}
    where the reminder $R^\varepsilon_p(u,\varphi)$  satisfies
    \begin{equation*}
        |R^\varepsilon_p(u,\varphi)| \leq\begin{cases}
         C_p (\varepsilon^2|u(x)-u(y)|^{p-2}|\varphi(x)-\varphi(y)|^2 + |\varepsilon|^p|\varphi(x)-\varphi(y)|^p), &\quad \text{ if } \ p >2,\\
         C_p |\varepsilon|^p|\varphi(x)-\varphi(y)|^p, &\quad \text{ if } \ p \leq 2,
  \end{cases}
    \end{equation*}
    and the reminder $R^\varepsilon_q(u,\varphi)$ satisfies
    \begin{equation*}
|R^\varepsilon_q(u,\varphi)| \leq  \begin{cases}
C_q (\varepsilon^2|u|^{q-2}|\varphi|^2 + |\varepsilon|^q|\varphi|^q), &\quad \text{ if } \ q >2,\\
C_q |\varepsilon|^q |\varphi|^q,& \quad \text{ if } \ q \leq 2.
    \end{cases}
    \end{equation*}
    Thus,
    \begin{equation}\label{star489tgjkwel8765}
        \begin{split}
        &J_\alpha(u+\varepsilon v)\\
        &= J_\alpha(u) + \varepsilon \Big(\frac{c_{n,s,p}}{2} \int_{\mathbb{R}^{2n}\setminus (\mathbb{R}^n\setminus \Omega)^2} \frac{|u(x)-u(y)|^{p-2}(u(x)-u(y))(\varphi(x)-\varphi(y))}{|x-y|^{n+sp}} \ dxdy\\
        &\qquad \qquad\qquad+ \alpha \int_{\mathbb{R}^n \setminus \Omega} |u|^{q-2}u\varphi - \int_\Omega f \varphi\Big)\\
        &\qquad\qquad +  \frac{c_{n,s,p}}{2p}\int_{\mathbb{R}^{2n}\setminus (\mathbb{R}^n\setminus \Omega)^2} \frac{R^\varepsilon_p(u,\varphi)}{|x-y|^{n+sp}} \ dxdy
        +  \frac{\alpha}{q} \int_{\mathbb{R}^n \setminus \Omega} R^\varepsilon_q(u,\varphi) \ dx .
    \end{split}
    \end{equation}
    
    In addition,
    as a consequence of the facts that $u$, $ \varphi \in \mathcal{W}^{s,p}$ and~$p$, $q>1$, and by H\"older's inequality, it follows that
    \begin{equation*}
        \lim_{\varepsilon \rightarrow 0} \frac{1}{\varepsilon}\int_{\mathbb{R}^{2n}\setminus (\mathbb{R}^n\setminus \Omega)^2}\frac{{R^\varepsilon_p(u,\varphi)}}{|x-y|^{n+sp}} dxdy= \lim_{\varepsilon \rightarrow 0}\frac1{\varepsilon} \int_{\mathbb{R}^n \setminus \Omega}{R^\varepsilon_q(u,\varphi)dx} =0.
    \end{equation*}
    Taking the limit as $\varepsilon \rightarrow 0$ in~\eqref{star489tgjkwel8765}, we thus deduce that
    \begin{equation*}
    \begin{split}
J_\alpha'(u)[\varphi]        &=\lim_{\varepsilon\rightarrow 0} \frac{J_\alpha(u+\varepsilon v) - J_\alpha(u)}{\varepsilon}\\ 
        &= \frac{c_{n,s,p}}{2} \int_{\mathbb{R}^{2n}\setminus (\mathbb{R}^n\setminus \Omega)^2} \frac{|u(x)-u(y)|^{p-2}(u(x)-u(y))(\varphi(x)-\varphi(y))}{|x-y|^{n+sp}} \ dxdy\\
        &\qquad \qquad\qquad+ \alpha \int_{\mathbb{R}^n \setminus \Omega} |u|^{q-2}u\varphi - \int_\Omega f \varphi,
    \end{split}
    \end{equation*}
   and accordingly, if $u \in \mathcal{W}^{s,p}$ is a critical point of the functional $J_\alpha$, then $u$ is a weak solution of~\eqref{pb:RF} in the sense of Definition~\ref{DEBOLE}.
\end{proof}

\section*{Acknowledgments}
S. Dipierro
is supported by the Australian Future Fellowship FT230100333. 
E. Valdinoci  is supported by the
Australian Laureate Fellowship FL190100081.
Work partially supported by PRIN PNRR P2022YFAJH ``Linear and Nonlinear PDEs: New directions and applications''. G. Spadaro is member of the INdAM-GNAMPA group (``Gruppo Nazionale per 
l'Analisi Matematica, la Probabilit\`a e le loro Applicazioni -- Istituto Nazionale di Alta
Matematica''). G. Spadaro thanks the University of Western Australia for the hospitality during the visiting period in which this work was developed.

\end{document}